\documentclass[english,twocolumn]{IEEEtran}
\usepackage{graphicx, footnote, amssymb, mathtools}
 \usepackage{graphicx}
 \usepackage{hyperref}
\usepackage{amsmath,epsfig,amssymb}
\usepackage{times,amsmath,epsfig,amssymb,graphicx,amsfonts}
\usepackage{color}
\usepackage{graphicx, footnote, amssymb, mathtools}
 \usepackage{graphicx}
\usepackage{amsmath,epsfig,amssymb}
\usepackage{times,amsmath,epsfig,amssymb,graphicx,amsfonts}
\usepackage{color}
\usepackage{graphicx, footnote, amssymb, mathtools}
 \usepackage{graphicx}
\usepackage{amsmath,epsfig,amssymb}
\usepackage{times,amsmath,epsfig,amssymb,graphicx,amsfonts}

\usepackage{color}
\usepackage{multirow}
\usepackage{psfrag}
\usepackage{amsmath}
\usepackage{amsthm}
\usepackage{epsfig}
\usepackage{subfig}
\usepackage{psfrag}
\usepackage{algorithm}
\usepackage{algorithmic}
 \newcommand{\bi}{\begin{itemize}}
\newcommand{\ei}{\end{itemize}}
\newcommand{\bal}{\begin{align}}
\newcommand{\eal}{\end{align}}

\newtheorem{theorem}{\textbf{Theorem}}
\newtheorem{lemma}{\textbf{Lemma}}

\newtheorem{example}{\textbf{Example}}
\newtheorem{definition}{\textbf{Definition}}

\newtheorem{remark}{\textbf{Remark}}
 
\newtheorem{proposition}{\textbf{Proposition}}
\newtheorem{assumption}{\textbf{Assumption}}

\begin{document}

\title{\huge{Convergence Rate of Distributed ADMM over Networks}}
\author{
\IEEEauthorblockN{Ali Makhdoumi and Asuman Ozdaglar}
\IEEEauthorblockA{
MIT,
Cambridge, MA 02139\\
Emails: makhdoum@mit.edu, asuman@mit.edu }
}
\maketitle
\begin{abstract}
We propose a distributed algorithm based on Alternating Direction Method of Multipliers (ADMM) to minimize the sum of locally known convex functions using communication over a network. This optimization problem emerges in many applications in distributed machine learning and statistical estimation. We show that when functions are convex, both the objective function values and the feasibility violation converge with rate $O(\frac{1}{T})$, where $T$ is the number of iterations. We then  show that if the functions are strongly convex and have Lipschitz continuous gradients, the sequence generated by our algorithm converges linearly to the optimal solution. In particular, an $\epsilon$-optimal solution can be computed with $O(\sqrt{\kappa_f} \log (1/\epsilon))$ iterations, where $\kappa_f$ is the condition number of the problem. Our analysis also highlights the effect of network structure on the convergence rate through maximum and minimum degree of nodes as well as the algebraic connectivity of the network.
\end{abstract}


\section{Introduction}
\subsection{Motivation}
Many of today's optimization problems in data science (including statistics, machine learning, and data mining) include an abundance of data, which cannot be handled by a single processor alone. This necessitates distributing data among multiple processors and processing it in a decentralized manner based on the available local information. The applications in machine learning  \cite{xiao2007distributed, dekel2012optimal, recht2011hogwild, agarwal2011distributed, duchi2014optimality} along with other applications in distributed data processing where information is inherently distributed among many processors (see e.g. distributed sensor networks \cite{duarte2004vehicle, rabbat2004distributed}, coordination and flow control problems \cite{low1999optimization, jadbabaie2003coordination}) have spearheaded a large literature on distributed multiagent optimization.

In this paper, we focus on the following optimization problem:
\begin{align}\label{eq:itrodproblemformula}
& \min_{x \in \mathbb{R}^d} \sum_{i=1}^n f_i(x) ,
\end{align}
where $f_i: \mathbb{R}^d \to \mathbb{R}$ is a convex function. We assume $f_i$ is known only to agent $i$ and refer to it as a local objective function.\footnote{We use the terms machine, agent, and node interchangeably.} Agents can communicate over a given network and their goal is to collectively solve this optimization. A prominent example where this general formulation emerges is  \emph{Empirical Risk Minimization} (EMR). Suppose that we have $M$ data points $\{(x_i,y_i)\}_{i=1}^M$, where $x_i \in \mathbb{R}^d$ is a feature vector and $y_i \in \mathbb{R}$ is a target output. The empirical risk minimization is then given by 
 \begin{align}\label{eq:EMR}
\min_{\theta \in \mathbb{R}^d} \frac{1}{M}\sum_{i=1}^M L(y_i, x_i, \theta)  + p(\theta),
\end{align}
for some convex loss function $L: \mathbb{R} \times \mathbb{R}^d \times \mathbb{R}^d \to \mathbb{R}$ and some convex penalty function $p: \mathbb{R}^d \to \mathbb{R}$. This general formulation captures many statistical scenarios including:
\begin{itemize}
\item Least-Absolute Shrinkage and Selection Operator (LASSO): \[\min_{\theta \in \mathbb{R}^d} \frac{1}{M} \sum_{i=1}^M (y_i - \theta' x_i)^2 + \tau ||\theta||_1.\]
\item Support Vector Machine (SVM) (\cite{cortes1995support}): \[\min_{\theta \in \mathbb{R}^d} \frac{1}{M} \sum_{i=1}^M \max\{0, 1- y_i (\theta' x_i)\}  + \tau ||\theta||_2^2.\]
\end{itemize}
Suppose our distributed computing system consists of $n$ machines each with $k= M/n$ data points (without loss of generality suppose $M$ is divisible by $n$, otherwise one of the machines have the remainder of data points). For all $i=1, \dots, n$, we define a function based on the available data to machine $i$ as  \[f_i(\theta)= \frac{1}{k} \sum_{1+(i-1)k}^{ik} L(y_i, x_i, \theta)  + p(\theta).\] 
Therefore, the empirical risk minimization  \eqref{eq:EMR} can be written as $\min_{\theta \in \mathbb{R}^d} \frac{1}{n}\sum_{i=1}^n f_i(\theta)$, 
where function $f_i(\theta)$ is only available to machine $i$, which is an instance of the formulation \eqref{eq:itrodproblemformula}. Data is distributed across different machines either because it is collected by decentralized agents \cite{zhang2013information, zhang2013divide, mitliagkas2013memory} or because memory constraints prevent it from being stored in a single machine\cite{duchi2014optimality, zhang2012communication, zhang2015communication, hellman1970learning}. The decentralized nature of data together with communication constraints necessitate distributed processing which has motivated a large literature in optimization and statistical learning on distributed algorithms (see e.g. \cite{forero2010consensus, li2014communication, wang2013large, aslan2013convex, romera2013new}).

\subsection{Related Works and Contributions}

Much of this literature builds on the seminal works \cite{tsitsiklis1984problems, tsitsiklis1986distributed}, which proposed gradient methods that can parallelize computations across multiple processors. 
A number of recent papers proposed subgradient type methods \cite{nedic2010constrained, ram2010distributed, chen2012diffusion, jakovetic2014fast, shi2014extra, ram2010distributed, johansson2009randomized} or a dual averaging method \cite{duchi2012dual} to design distributed optimization algorithms.

An alternative approach is to use Alternating Direction Method of Multipliers (ADMM) type methods which for separable problems leads to decoupled computations (see e.g. \cite{boyd2011distributed} and \cite{ eckstein2012augmented} for comprehensive tutorials on ADMM). ADMM has been studied extensively in the 80's \cite{glowinski1975approximation, gabay1976dual, bertsekas1988dual}. More
recently, it has found applications in a variety of distributed settings in machine learning such as model fitting, resource allocation, and classification (see e.g.  \cite{wahlberg2012admm, sedghi2014multi, wang2012online, zhang2012efficient, zhang2014asynchronous, mota2011basis, schizas2008consensus, aybat2014alternating, aybat2014admm}).

In this paper we present a new distributed ADMM algorithm for solving problem \eqref{eq:itrodproblemformula} over a network. Our algorithm relies on a novel node-based reformulation of \eqref{eq:itrodproblemformula} and leads to an ADMM algorithm that uses dual variables with dimension given by the number of nodes in the network. This results in a significant reduction in the number of variables stored and communicated with respect to edge-based ADMM algorithm presented in the literature (see \cite{ShiYin2014LinADMM, ermin}). Our main contribution is a unified convergence rate analysis for this algorithm that applies to both the case when the local objective functions are convex and also the case when the local objective functions are strongly convex with Lipschitz continuous gradients. In particular, our analysis shows that when the local objective functions are convex (with no further assumptions), then the objective function at the ergodic average of the estimates generated by our algorithm converges with rate $O(\frac{1}{T})$. Moreover, when the local objective functions are strongly convex with Lipschitz continuous gradients we show that the iterates converges linearly, i.e., the iterates converge to an $\epsilon$-neighborhood of the optimal solution after $O(\sqrt{\kappa_f} \log(\frac{1}{\epsilon}))$ steps, where $\kappa_f$ is the condition number defined as $L/\nu$, where $L$ is the maximum Lipschitz gradient parameter and $\nu$ is the minimum strong convexity constant of the local objective functions. This matches the best known iteration complexity and condition number dependence for the centralized ADMM (see e.g. \cite{deng2012global}). Our convergence rate estimates also highlight a novel dependence on the network structure as well as the communication weights. In particular, for communication weights that are governed by the Laplacian of the graph, we establish a novel iteration complexity $O\left(\sqrt{\kappa_f} \sqrt{ \frac{d^4_{\text{max}}}{d_{\text{min}}a^2(G) } } \log \left( \frac{1}{\epsilon} \right) \right)$, where $d_{\text{min}}$ is the minimum degree, $d_{\text{max}}$ is the maximum degree, and $a(G)$ is the algebraic connectivity of the network. Finally, we illustrate the performance of our algorithm with numerical examples. 

Our paper is most closely related to \cite{ShiYin2014LinADMM, ermin}, which studied edge-based ADMM algorithms for solving \eqref{eq:itrodproblemformula}. In \cite{ermin}, the authors consider convex local objective functions and provide an $O(\frac{1}{T})$ convergence rate. The more recent paper \cite{ShiYin2014LinADMM} assumes strongly convex local objective functions with Lipschitz gradients and show a linear convergence rate through a completely different analysis. This analysis does not extend to the node-based ADMM algorithm under these assumptions. In contrast, our paper provides a unified convergence rate analysis for both cases for the node-based distributed ADMM algorithm. Our paper is also related to \cite{he20121, deng2012global} that study the basic centralized ADMM where the goal is to miminize sum of two functions with a linearly coupled constraint. Our work is also related to the literature on the converge of operator splitting schemes, such as Douglas-Rachford splitting and relaxed Peaceman-Rachford
\cite{lions1979splitting, eckstein1998operator, patrinos2014douglas, patrinos2014forward, davis2014convergence, davis2014convergence2, giselsson2014diagonal, nishihara2015general, lessard2014analysis}.

\subsection{Outline}
The organization of paper is as follows. In Section \ref{sec:framework} we give the problem formulation and propose a novel distributed ADMM algorithm. In Section \ref{sec:preliminary}, we show some preliminary results that helps us to show the main results. In Section \ref{sec:sublinearrate} we show the sub-linear convergence rate. In Section \ref{sec:linearrate} we show the linear convergence rate of our algorithm. Finally, in Section \ref{sec:numericalresults} we show the effect of network on the convergence rate and provide numerical results that illustrate the performance of our algorithm, which leads to concluding remarks in Section \ref{sec:conclusion}.  All the omitted proofs are presented in the appendix.

\section{Framework}\label{sec:framework}
\subsection{Problem Formulation}\label{sec:problemformulation}
Consider a network represented by a connected graph $G=(V, E)$ where $V=\{1, \dots, n\}$ is the set of agents and $E$ is the set of edges. For any $ i $, let $N(i)$ denote its set of neighbors including agent $i$ itself, i.e., $N(i)=\{j ~ | ~ (i,j) \in E\} \cup \{i\}$, and let $d_i$ denote the degree of agent $i$, i.e., $|N(i)|=d_i+1$. We let $d_{\text{max}}= \max_{i \in V} d_i$ and $d_{\text{min}}= \min_{i \in V} d_i$.

The goal of the agents is to collectively solve optimization problem \eqref{eq:itrodproblemformula}, 
where $f_i$ is a function known only to agent $i$. In order to solve optimization problem \eqref{eq:itrodproblemformula}, we introduce a variable $x_i$ for each $i$ and write the objective function of problem \eqref{eq:itrodproblemformula} as $\sum_{i=1}^n f_i(x_i)$, so that the objective function is decoupled across the agents. The constraint that all the $x_i$'s are equal can be imposed using the following matrix.

\begin{definition}[\textbf{Communication Matrix}]\label{def:Amatrix}
\textup{
Let $P$ be a $n \times n$ matrix whose entries satisfy the following property:
\\ For any $ i= 1, \dots,  n$, $P_{ij}=0$ for $j \notin N(i)$. We refer to $P$ as the {\it communication matrix}. 
}
\end{definition}

\begin{assumption}\label{assump:matrixA}
\textup{ The communication matrix $P$ satisfies $\text{null}(P)=\text{span}\{\mathbf{1}\}$, where $\mathbf{1}$ is a $n \times 1$ vector with all entries equal to one and $\text{null}(P)$ denotes the null-space of the matrix $P$. 
}
\end{assumption}
\begin{example}
\textup{If $P_{ij}<0$ for all $j \in N(i)\setminus \{i\}$, summation of each row of $P$ is zero, and the graph is connected, then Assumption \ref{assump:matrixA} holds. As a particular case, the Laplacian matrix of the graph given by $P_{ij}= -1$ when $j \in N(i)\setminus \{i\}$ and zero otherwise, and $P_{ii}=d_i$ is a communication matrix that satisfies Assumption \ref{assump:matrixA}.
}
\end{example}
 We next show that the constraint that all $x_i$'s are equal can be enforced by the linear constraint $A \mathbf{x}=0$, where $\mathbf{x}=(x_1, \dots, x_n)$ where each $x_i$ is a sub-vector of dimension $d$ and $A$ is a $dn \times dn$ matrix defined as the Kronecker product between communication matrix $P$ and $I_d$, i.e., $A= P \otimes  I_d$. 

\begin{lemma}\label{lem:constraint}
Under Assumption \ref{assump:matrixA}, the constraint $A \mathbf{x}=0$ guarantees that $x_i=x_j$ for all $i,j \in V$.
\end{lemma}

Using Lemma 1, under Assumption \ref{assump:matrixA}, we can reformulate optimization problem \eqref{eq:itrodproblemformula} as 
\begin{align}\label{eq:optlinearconstraint}
& \min_{\textbf{x} \in {\mathbb{R}}^{nd}} F(\mathbf{x}) \\
& \text{ s.t. } A \mathbf{x} =0, \nonumber
\end{align}
where  $F(\mathbf{x})=\sum_{i=1}^n f_i(x_i)$.
\begin{assumption}\label{assump:nonempty-set}
\textup{The optimal solution set of problem \eqref{eq:optlinearconstraint} is non-empty. We let $\mathbf{x}^*$ denote an optimal solution of the problem \eqref{eq:optlinearconstraint}.
}
\end{assumption}

\subsection{Multiagent Distributed ADMM}\label{sec:algorithm}
In this section, we propose a distributed ADMM algorithm to solve {problem} \eqref{eq:optlinearconstraint}. 
{We first use a reformulation technique (this technique was introduced in \cite{bertsekas1989parallel} to separate optimization variables in a constraint, allowing them to be updated simultaneously in an ADMM iteration), which allows us to separate each constraint associated with a node into  multiple constraints that involve only the variable corresponding to one of the neighboring nodes}. We expand the constraint $A \mathbf{x}=0$  so that for each  node $i$, we have $\sum_{j \in N(i)} A_{ij} x_j=0$,
where $A_{ij}=P_{ij} \otimes I_d$ is a $d \times d$ matrix. 
We let $A_{ij} x_j=z_{ij} \in \mathbb{R}^d$ to obtain the following reformulation: 
\begin{align}\label{eq:optreformulation}
& \min_{\mathbf{x}, ~ \mathbf{z} } F(\mathbf{x})  \nonumber \\
& \text{ s.t. } A_{ij} x_j = z_{ij}, ~~ \text{ for } i=1, \dots, n, ~ j \in N(i), \nonumber \\
& ~~~~ \sum_{j \in N(i)} z_{ij}=0, ~ \text{ for } i=1, \dots, n. 
\end{align}
For each equality constraint in \eqref{eq:optreformulation}, we let $\lambda_{ij} \in \mathbb{R}^d$ be the corresponding Lagrange multiplier and form the augmented Lagrangian function by adding a quadratic 
penalty with penalty parameter $c>0$ for feasibility violation to the Lagrangian function as 
\begin{align*}
L_c(\mathbf{x}, \mathbf{z},\mathbf{\lambda}) = F(\mathbf{x})  & + \sum_{i=1}^n \sum_{j \in N(i)} \lambda_{ij}' (A_{ij} x_j - z_{ij})  \\
& + \frac{c}{2} \sum_{i=1}^n \sum_{j \in N(i)} ||A_{ij} x_j - z_{ij}||_2^2. 
\end{align*}

We now use ADMM algorithm (see e.g. \cite{bertsekas2003convex}). ADMM algorithm generates  primal-dual sequences $\{x_j(t)\}, \{z_{ij}(t)\}$, and $\{\lambda_{ij}(t)\}$ which at iteration $t+1$ are  updated as follows: 

\begin{enumerate}
\item For any $j= 1, \dots, n$, we update $x_j$ as 
\begin{align}\label{eq:temp0}
& x_j(t+1)\in  \text{argmin}_{x_j \in \mathbb{R}^d}  L_c(\mathbf{x}, \mathbf{z}(t), \mathbf{\lambda}(t)). 
\end{align}
\item For any $i= 1, \dots, n$, we update the vector $\mathbf{z}_i=[z_{ij}]_{j\in N(i)}$ as
\begin{align}\label{eq:step2}
& \mathbf{z}_{i}(t+1) \in \text{argmin}_{\mathbf{z}_{i} \in Z_i}  L_c(\mathbf{x}(t+1), \mathbf{z}, \mathbf{\lambda}(t)). , 
\end{align}
where $Z_i=\{\mathbf{z}_i\ |\ \sum_{j \in N(i)} z_{ij}=0\}$.
\item For $i= 1, \dots, n$ and $ j \in N(i)$ we update $\lambda_{ij}$ as
\begin{align}\label{eq:temp3}
\lambda_{ij}(t+1)= \lambda_{ij}(t)+ c (A_{ij} x_j(t+1) - z_{ij} (t+1)).
\end{align} 
\end{enumerate}
One can implement this algorithm in a distributed manner, where node $i$ maintains variables $\lambda_{ij}(t)$ and $z_{ij}(t)$ for all $j \in N(i)$ (\cite{ermin}).  However, using the inherent symmetries in the problem, we can significantly reduce the number of variables that each node requires to maintain from $O(|E|)$ to $O(|V|)$. 

We first show that for all $t$, $i$, and $j \in N(i)$, we have $\lambda_{ij}(t)= p_i (t)$. This reduction shows that the algorithm need not maintain dual variables $\lambda_{ij}(t)$ for each $i$ and its neighbors $j$, but instead can operate with the lower dimensional node-based dual variable $p_i(t)$. The dual variable $p_i(t)$ can be updated  using primal variables $x_j(t)$ for all $j \in N(i)$. The second observation is that $z_{ij}(t)= A_{ij} x_j(t) - y_i(t)$, where $y_i(t) = \frac{1}{d_i + 1} \left({[A]^i}\right)' \mathbf{x}(t)$ and $\left({[A]^i}\right)'= (P_{i1}, \dots, P_{in}) \otimes I_d$. This reduction shows that the algorithm need not maintain primal variables $z_{ij}(t)$ for each $i$ and its neighbors $j$, but instead can operate with the lower dimensional node-based primal variables $y_i(t)$, where $y_i(t)$ is node $i$'s estimate of the primal variable (obtained as the average of primal variables of his own neighbors). The aforementioned reductions are shown in the following proposition.

\begin{proposition}\label{pro:broadcast}
The sequence $\{x_i(t)\}_{t=0}^{\infty}$ for $i=1, \dots, n$ generated by implementing the steps presented in Algorithm \eqref{alg:ADMM} is the same as the sequence generated by the ADMM algorithm. 
\end{proposition}

The steps of the algorithm can be implemented in a distributed way, meaning that each node first updates her estimates based on the information received from her neighboring nodes and then broadcasts her updated estimates to her neighboring nodes. Each node $i$ maintains local variables $x_i(t)$, $y_i(t)$, and $p_i(t)$ and updates these variables using communication with its neighbors as follows:
\begin{itemize}
\item At the end of iteration $t$, each node  $i$ sends out $p_i(t)$ and $y_i(t)$ to all of its neighbors and then each  node such as $j$ uses $y_i(t)$ and $p_i(t)$ of all $i \in  N(j)$ to update $x_{j}(t+1)$ as in step 1. 
\item Each node $j$ sends out $x_j(t+1)$ to all of its neighbors and then each node such as $i$ computes $y_i(t+1)$ as in step 2.
\item Each node $i$ updates $p_i(t+1)$ as in step 3. 
\end{itemize}
Using this algorithm agent $i$ need to store only three variables, $x_i(t)$, $y_i(t)$, and $p_i(t)$ and update them at each iteration. Also, each agent need to communicate only with (broadcast her estimates to) its neighbors. Therefore, the overall storage requirement is $3 |V|$ and the overall communication requirement at each iteration is $|E|$. 

\begin{algorithm}[t] 
\caption{Multiagent Distributed ADMM} \label{alg:ADMM}
\begin{itemize}
\item \textbf{Initialization:} $x_i(0)$, $y_i(0)$, and $p_i(0)$ all in $\mathbb{R}^d$, for any $ i \in V$ and matrix $A \in \mathbb{R}^{nd \times nd}$.  
\item \textbf{Algorithm:}
\begin{enumerate}
\item for $i = 1, \dots, n$, let
\begin{align*}
x_i(t+1) \in \text{argmin}_{x_i \in \mathbb{R}^d} & f_i(x_i)  + \sum_{j \in N(i)} ( p'_j(t) A_{ji} x_i \nonumber \\
& + \frac{c}{2} || y_j(t) + A_{ji} (x_i-x_i(t) ||_2^2). 
\end{align*}
\item for $i = 1, \dots, n$, let
\begin{align*}
y_i(t+1) = \frac{1}{d_i+1}\sum_{j \in N(i)} A_{ij} x_{j}(t+1).
\end{align*}
\item for $i = 1, \dots, n$, let 
\begin{align*}
{p}_i(t+1) = {p}_i(t) + c {y}_i(t+1)
\end{align*}
\end{enumerate}
\item \textbf{Output:} $\{x_i(t)\}_{t=0}^{\infty}$ for any $i \in V$. 
\end{itemize}
\end{algorithm}

\section{Preliminary Results}\label{sec:preliminary}
In this section, we present the preliminary results that we will use to establish our convergence rate.
we define 
\begin{align*}
\partial F(\mathbf{x}) &= \{h \in \mathbb{R}^{nd}~ :~  h=(h_1(x_1)', \dots, \nabla h_n(x_n)')'\\
&  ,~ h_i(x_i) \in \partial f_i(x_i) \}, 
\end{align*}
where for each $i$, $\partial f_i(x_i)$ denotes subdifferential of $f_i$ at $x_i$, i.e., the set of all subgradients of $f_i$ at $x_i$. In what follows, for notational simplicity we assume  $d=1$, i.e., in \eqref{eq:optlinearconstraint} $x \in \mathbb{R}$. All the analysis generalizes to the case with $x \in \mathbb{R}^d$. 
We first provide a compact representation of the evolution of primal vector $\mathbf{x}(t)$ that will be used in the convergence proof. This is a core step in proving the convergence rate as it eliminates the dependence on the other variables $y_i(t)$ and $p_i(t)$. Let $M$ be a $n \times n$ diagonal matrix with $M_{ii}= \sum_{j \in N(i)} A_{ji}^2$ and $D$ be a $n \times n$ diagonal matrix with $D_{ii}= d_i+1$. 
\begin{lemma}[\textbf{Perturbed Linear Update}]\label{lem:purturbed}
The update of Algorithm \ref{alg:ADMM} can be written as 
\begin{align*}
\mathbf{x}(t+1)  = & -\frac{1}{c} M^{-1} h(\mathbf{x}(t+1)) + \left( I - M^{-1} A' D^{-1} A \right) \mathbf{x}(t) \nonumber\\
& - M^{-1} (A' D^{-1} A) \sum_{s=0}^t \mathbf{x}(s),
\end{align*}
for some $h(\mathbf{x}(t+1)) \in \partial F(\mathbf{x}(t+1))$. 
\end{lemma}
Lemma \autoref{lem:purturbed} shows $\mathbf{x}(t+1)$ can be written as a perturbed linear combination of $\{\mathbf{x}(s)\}_{s=0}^t$ with the perturbation being the term  $-\frac{1}{c} M^{-1} h(\mathbf{x}(t+1))$. The intuition behind the convergence rate analysis is that the linear term that relates $\mathbf{x}(t+1)$ to $\mathbf{x}(0), \dots, \mathbf{x}(t)$ guarantees that the sequence $\mathbf{x}(t)$ converges to a consensus point where $x_i(t)= x_j(t)$ for all $i, j \in V$; and the perturbation term $-\frac{1}{c} M^{-1} h(\mathbf{x}(t+1))$ guarantees that the converging point minimizes the objective function $F(\mathbf{x})= \sum_{i=1}^n f_i(x_i)$.

\section{Sub-linear Rate of Convergence}\label{sec:sublinearrate}
In this section, we show the sublinear rate of convergence. We define two auxiliary sequences that we will use in proving the convergence rates.
Since $A'D^{-1}A$ is positive semidefinite (see Lemma \ref{lem:PSD} in the appendix), we can define $Q= (A'D^{-1}A)^{1/2}$. In other words, we let $Q= V\Sigma^{1/2} V'$, where $A'D^{-1}A = V\Sigma V'$ is the singular value decomposition of the symmetric matrix $A'D^{-1}A$.  We define the auxiliary sequences  \[\mathbf{r}(t) = \sum_{s=0}^t Q \mathbf{x}(s)  ,\]
and 
\[\mathbf{q}(t)= \begin{pmatrix}
  \mathbf{r}(t)  \\
  \mathbf{x}(t) 
 \end{pmatrix}. \] We also let \[G= \begin{pmatrix}
  I & 0  \\
  0 & M- A'D^{-1}A 
 \end{pmatrix}.\] 
 Next, we show a proposition that bounds the function value at each iteration. 
\begin{proposition}\label{pro:telescopic}
For any $\mathbf{r} \in \mathbb{R}^d$ and $t$, the sequence generated by Algorithm \ref{alg:ADMM} satisfies:
\begin{align*}
& \frac{2}{c}\left(F(\mathbf{x}(t+1))- F(\mathbf{x}^*)  \right)+ 2 \mathbf{r}' Q \mathbf{x}(t+1)  \nonumber\\ & \le ||\mathbf{q}(t)- \mathbf{q}^*||_G^2  - ||\mathbf{q}(t+1)- \mathbf{q}^*||_G^2 - ||\mathbf{q}(t)- \mathbf{q}(t+1)||_G^2 ,
\end{align*}
where $\mathbf{q}^*= \begin{pmatrix}
  \mathbf{r}^*  \\
  \mathbf{x}^* 
 \end{pmatrix}$. 
\end{proposition}
In order to obtain $O(1/T)$ convergence rate, we consider the performance of the algorithm at the ergodic vector defined as $\hat{\mathbf{x}}(T) = (\hat{x}_1 (T), \dots, \hat{x}_n (T))$, where 
$$\hat{{x}}_i(T)= \frac{1}{T} \sum_{t=1}^{T} {x}_i(t),$$ for any $i= 1, \dots, n$. Note that each agent $i$ can construct this vector by simple recursive time-averaging of its estimate $x_i(t)$. 
Let $(\mathbf{x}^*, \hat{\mathbf{r}})$ be a primal-dual optimal solution of 
\begin{align*}
\min_{Q \mathbf{x}=0} F(\mathbf{x}).
\end{align*}
Since $\text{null}(Q)=\text{null}(P)$, under Assumption \eqref{assump:matrixA}, the optimal primal solution of this problem is the same as of the original problem \eqref{eq:optlinearconstraint}
Next, we show both objective function and feasibility violation converges with rate $O(\frac{1}{T})$ to the optimal value. 
\begin{theorem}\label{thm:sublinearrate}
For any $T$, we have  
\begin{align*}
& | F(\hat{\mathbf{x}}(T)) - F(\mathbf{x}^*)|  \le \frac{c}{2T} \left(||\mathbf{x}(0)-\mathbf{x}^*||_{M- A' D^{-1} A}^2 \right) \\
& + \frac{c}{2T} \left( \max\{ ||\mathbf{r}(0)- 2 \hat{\mathbf{r}}||_2^2, ||\mathbf{r}(0)||_2^2 \} \right).
\end{align*}
We also have  
\begin{align*}
&   ||Q \hat{\mathbf{x}}(T)||_2  \le  \frac{1}{2T} \left(||\mathbf{x}(0)-\mathbf{x}^*||_{M- A' D^{-1} A}^2 \right) \\
& + \frac{1}{2T} \left( 2||\mathbf{r}(0)-  \hat{\mathbf{r}}||_2^2+2 \right).
\end{align*}
\end{theorem}
This theorem shows that the objective function at the ergodic average of the sequence of estimates generated by Algorithm \ref{alg:ADMM} converges with rate $O(\frac{1}{T})$ to the optimal solution. We next characterize the network effect on the performance guarantee.  
\begin{theorem}\label{thm:sublinearnetworkdependence}
For any $T$, starting form $\mathbf{x}(0)=0$, we have  
\begin{align*}
 | F(\hat{\mathbf{x}}(T)) - F(\mathbf{x}^*)| \le & \frac{c}{2T} ||\mathbf{x}^*||_2^2 \lambda_M +  \frac{2}{c T} \frac{U^2}{\tilde{\lambda}_m}, 
\end{align*}
and 
\begin{align*}
   ||Q \hat{\mathbf{x}}(T)|| \le  \frac{1}{2T} ||\mathbf{x}^*||_2^2 \lambda_M  + \frac{1}{2T} \left( 2+ 2\frac{U^2}{c^2 \tilde{\lambda}_m} \right), 
\end{align*}
where $U$ is a bound on the subgradients of the function $F$ at $\mathbf{x}^*$, i.e., $\|\mathbf{v}\|\le U$ for all $\mathbf{v}\in \partial F(\mathbf{x}^*)$,  
$\tilde{\lambda}_m$ is the smallest non-zero eigen value of $A' D^{-1} A$, and $\lambda_M$ is the largest eigen value of $M-A' D^{-1} A$. 
\end{theorem}
\begin{remark}
\textup{
Both the optimality of the objective function value at the ergodic average and the feasibility violation converge with rate $O(\frac{1}{T})$. Our guaranteed rates show a novel dependency on the network structure and communication matrix through  $\tilde{\lambda}_m$ and $\lambda_M$. Therefore, for a given function, in order to obtain a better performance guarantee we need to maximize $\tilde{\lambda}_m$ and minimize $\lambda_M$. In Section \eqref{sec:numericalresults} we show that these terms depend on the algebraic connectivity of the network and provide explicit dependencies solely on the network structure when the communication matrix is the Laplacian of the graph.
}
\end{remark}
\section{Linear Rate of Convergence}\label{sec:linearrate}
In order to show the linear rate of convergence, we adopt the following standard assumptions. 
\begin{assumption}[\textbf{Strongly convex and Lipschitz Gradient}]\label{assump:strnogconvnlipschittzgard}
\textup{
For any $i=1, \dots, n$, the function $f_i$ is  differentiable and has Lipschitz continuous gradient, i.e., 
\begin{align*}
| \nabla f_i(x)- \nabla f_i(y)| \le L_{f_i} ||x-y||_2, \text{  for any  } x, y \in \mathbb{R}^d,
\end{align*}
for some $L_{f_i} \ge 0$. The function $f_i$ is also strongly convex with parameter $\nu_{f_i} > 0$, i.e., $f_i(x)- \frac{\nu_{f_i}}{2} ||x||_2^2$ is convex. 
}
\end{assumption}
We let $\nu= \min_{1 \le i \le n} \nu_{f_i}$ and $L= \max_{1 \le i \le n} L_{f_i}$, and define the condition number of $F(\mathbf{x})$ (or the condition number of problem  \eqref{eq:optlinearconstraint}) as $\kappa_f=\frac{L}{\nu}$. Note that when the functions are differentiable, we have 
\begin{align*}
\nabla F(\mathbf{x})= (\nabla f_1(x_1)', \dots, \nabla f_n(x_n)')' \in \mathbb{R}^{nd}.
\end{align*}

Assumption \autoref{assump:strnogconvnlipschittzgard} results in the following standard inequalities for the aggregate function $F(\mathbf{x})$. 
\begin{lemma}\label{lem:strongconvexandlipschitz} 
\begin{itemize}
\item [(a)] Under Assumption \ref{assump:strnogconvnlipschittzgard}, for any $\mathbf{x}, \mathbf{y} \in \mathbb{R}^{nd}$, we have \[(\nabla F(\mathbf{x})- \nabla F(\mathbf{y}))' (\mathbf{x}- \mathbf{y}) \ge \nu ||\mathbf{x}- \mathbf{y}||_2^2.\]
\item [(b)] Under Assumption \ref{assump:strnogconvnlipschittzgard}, for any $\mathbf{x}, \mathbf{y} \in \mathbb{R}^{nd}$, we have \[(\nabla F(\mathbf{x})- \nabla F(\mathbf{y}))' (\mathbf{x}- \mathbf{y}) \ge \frac{1}{L} ||\nabla F(\mathbf{x})- \nabla F(\mathbf{y})||_2^2.\]
\item [(c)] Under convexity assumption, for any $\mathbf{x}, \mathbf{y} \in \mathbb{R}^{nd}$ and $h(\mathbf{x}) \in \partial F(\mathbf{x})$ we have \[(\mathbf{x}- \mathbf{y})' h(\mathbf{x})  \ge F(\mathbf{x})- F(\mathbf{y}). \]
\end{itemize}

\end{lemma}
Under Assumption \eqref{assump:strnogconvnlipschittzgard} we show that the sequence generated by Algorithm \eqref{alg:ADMM} converges linearly to the optimal solution (which is unique under these assumptions). The idea is to use strong convexity and Lipschitz gradient property of $F(\mathbf{x})$ in order to show that the $G$-norm of sequence $\mathbf{q}(t)- \mathbf{q}^*$ contracts at each iteration, providing a linear rate. 
\begin{theorem}\label{thm:linearconvergence}
Suppose Assumptions \autoref{assump:matrixA}, \autoref{assump:nonempty-set}, and  \autoref{assump:strnogconvnlipschittzgard} hold. For any value of the penalty parameter $c>0$ and $\beta \in (0,1)$, the sequence generated by Algorithm \autoref{alg:ADMM} $\{\mathbf{x}(t)\}_{t=1}^{\infty}$ satisfies
\begin{align*}
||\mathbf{x}(t)- \mathbf{x}^*||_2^2  \le \left( \frac{1}{1+\delta} \right)^{t} ||\mathbf{q}(0)- \mathbf{q}^*||_2^2,
\end{align*}
where \[\delta \le \min \left \{ \frac{2 \beta \nu }{c \lambda_M (1+ \frac{2}{\tilde{\lambda}_m})}, \frac{(1-\beta) c \tilde{\lambda}_m }{L}\right \},\]
and $\tilde{\lambda}_{m}$ is the smallest non-zero eigen value of $A'D^{-1}A$, $\lambda_{M}$ is the largest eigen value of $M-A'D^{-1}A$. 
\end{theorem}

The rate of convergence in Theorem \ref{thm:linearconvergence} holds for any choice of penalty parameter $c >0$. In other words, for any choice of $c > 0$, the convergence rate is linear. We now optimize the rate of convergence over all choices of $c$ and provide an explicit convergence rate estimate that highlights dependence on the condition number of the problem. 
\begin{theorem} \label{thm:optimalstepsize}
Suppose Assumptions \autoref{assump:matrixA}, \autoref{assump:nonempty-set}, and  \autoref{assump:strnogconvnlipschittzgard} hold. Let $\{\mathbf{x}(t)\}_{t=1}^{\infty}$ be the sequence generated by Algorithm \autoref{alg:ADMM}. There exist $c> 0$ for which we have 
\begin{align*}
||\mathbf{x}(t)- \mathbf{x}^*||_2^2  \le   \rho ^{t} ||\mathbf{q}(0)- \mathbf{q}^*||_2^2,
\end{align*}
where the rate $\rho <1$ is given by \[\rho=  \left(1+ \frac{1}{2} \sqrt{\frac{2 \tilde{\lambda}_m^2}{ \lambda_M (2 + \tilde{\lambda}_m)}\frac{1}{\kappa_f}} \right)^{-1}.\] 
\end{theorem}
\begin{remark}
\textup{
This result shows that within $O\left( \sqrt{\kappa_f} \log \left( \frac{1}{\epsilon} \right) \right)$ iterations, the estimates $\{\mathbf{x}(t)\}$ reach an $\epsilon$- neighborhood of the optimal solution. Our rate estimate has a $\sqrt{\kappa_f}$ dependence which improves on the linear condition number dependence provided in the convergence analysis of edge-based ADMM in \cite{deng2012global}. The network dependence in our rate estimates is captured through $\tilde{\lambda}_m$ and $\lambda_M$. In particular, the larger $\tilde{\lambda}_m$ and the smaller $\lambda_M$ results in a faster rate of convergence. In Section \eqref{sec:numericalresults} we will explicitly show the network effect in the convergence rate and provide numerical results that illustrate the performance for networks with different connectivity properties. 
}
\end{remark}
\section{Network Effects}\label{sec:numericalresults}
We can choose communication matrix $P$ (and the corresponding matrix $A$) in the Algorithm \ref{alg:ADMM} to be any matrix that satisfies Assumption \eqref{assump:matrixA}. One natural choice for the matrix $P$ is the Laplacian of the graph which  leads to having $A_{ij}=A_{ji}=-1$ for all $j \in N(i) \setminus \{i\}$ and $A_{ii}=d_i$. Using Laplacian as the communication matrix we can now capture the effect of network structure in the convergence rate.

\subsection{Network Effect in Sub-linear Rate}
The following proposition explicitly show the networks dependence in the bounds provided in \autoref{thm:sublinearnetworkdependence}. 
\begin{proposition}\label{pro:sublinearexplicitnetworkdep}
For any $T$, starting form $\mathbf{x}(0)=0$ and using standard Laplacian as the communication matrix, we have  
\begin{align*}
 | F(\hat{\mathbf{x}}(T)) - F(\mathbf{x}^*)| \le & \frac{c}{2T} ||\mathbf{x}^*||_2^2 \left( 4 d_{\text{max}}^2 \right) +  \frac{2}{c T} U^2 \frac{2 d_{\text{max}}}{a(G)^2}, 
\end{align*}
and 
\begin{align*}
   ||Q \hat{\mathbf{x}}(T)|| \le  \frac{1}{2T} ||\mathbf{x}^*||_2^2 \left( 4 d_{\text{max}}^2 \right)   + \frac{1}{2T} \left( 2+ \frac{2 U^2}{ c^2} \frac{2 d_{\text{max}}}{a(G)^2} \right),
\end{align*}
where $U$ is a bound on the subgradients of the function $F$ at $\mathbf{x}^*$, i.e., $\|\mathbf{v}\|\le U$ for all $\mathbf{v}\in \partial F(\mathbf{x}^*)$ and $a(G)$ is the algebraic connectivity of the graph. 
\end{proposition}
Therefore, highly connected graphs with larger algebraic connectivity has a faster convergence rate  (see e.g. \cite{chung1997spectral}, \cite{fiedler1973algebraic} for an overview of the results on algebraic connectivity). 
\subsection{Network Effect in Linear Rate}
The following proposition  explicitly show the networks dependence in the bound provided in \autoref{thm:optimalstepsize}. 
\begin{proposition}\label{pro:linearexplicitnetworkdep}
Suppose Assumptions \autoref{assump:matrixA}, \autoref{assump:nonempty-set}, and  \autoref{assump:strnogconvnlipschittzgard} hold. Using standard Laplacian as the communication matrix, in order to reach an $\epsilon$-optimal solution $O\left(\sqrt{\kappa_f} \sqrt{ \frac{d^4_{\text{max}}}{d_{\text{min}}a^2(G) } } \log \left( \frac{1}{\epsilon} \right) \right)$ iterations suffice. 
\end{proposition}
Both of our guaranteed rates for sub-linear and linear rates depends on three parameters $d_{\text{max}}$, $d_{\text{min}}$ and $a(G)$. The convergence rate is faster for larger  $d_{\text{min}}$ and smaller $d_{\text{max}}$. Finally, the convergence rate is faster for larger algebraic connectivity $a(G)$.
\begin{example}
\textup{
To provide more intuition on the networks  dependence, we focus on $d$-regular graphs with matrix $P$ equal to Laplacian of the graph. In this setting, we have: $\tilde{\lambda}_m=\frac{a(G)^2}{d+1}$ and $\lambda_M=d(d+ 1)$, where $a(G)$ is the algebraic connectivity of the graph. Thus in this case the iteration complexity is $O\left(\sqrt{\kappa_f} \sqrt{ \frac{d^3}{a^2(G) }} \log \left( \frac{1}{\epsilon} \right) \right)$ (note that this bound matches the one provided in Proposition \autoref{pro:linearexplicitnetworkdep}). For $d$-regular graphs there exist good expanders such as Ramanujan graphs for which $a(G)=O(d)$ (see e.g. \cite{de2007old}). 
 In \autoref{fig:regularcompare}, we compare the performance of our algorithm for several regular graphs. The choice of function is $F(x)=\frac{1}{2} \sum_{i=1}^n (x-a_i)^2$ where $a_i$ is a scalar that is known only to machine $i$ (where $a_i=i$ for $i=1, \dots, n$). The communication matrix used in these experiments is the Laplacian of the graph. This problem appears in distributed estimation where the goal is to estimate the parameter $x^*$, using local measurements $a_i=x^*+N_i$ at each machine $i=1, \dots, n$. Here $N_i$ represents measurements noise, which we assume to be jointly Gaussian with mean zero and variance one. The maximum likelihood estimate is the minimizer $x^*$ of $F(x)$. 
 }
 \end{example}
\begin{figure}[t]
\centering
  \includegraphics[width=.9 \linewidth]{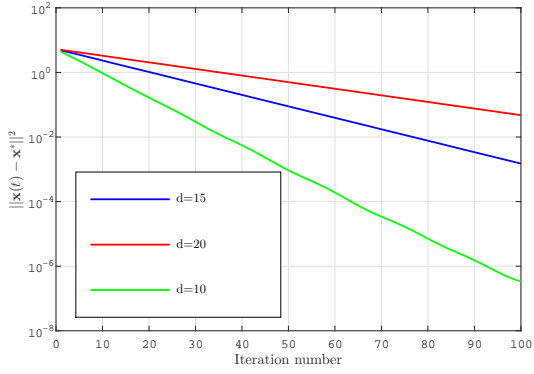}
  \caption{Performance of Algorithm \autoref{alg:ADMM} for three $d$-regular graphs with $d=10, 20, 30$. The $y$ axis is logarithmic to show the linear convergence rate.}
  \label{fig:regularcompare}
\end{figure}
\section{conclusion}\label{sec:conclusion}
We proposed a novel distributed algorithm based on Alternating  Direction Method of Multipliers (ADMM) to minimize the sum of locally known convex functions. We first showed that ADMM can be implemented by only keeping track of some node-based variables. We then showed that our algorithm reaches $\epsilon$-optimal solution in $O \left( \frac{1}{\epsilon} \right)$ number of iterations for convex functions and in $\left( \sqrt{\kappa_f} \log\left(\frac{1}{\epsilon}\right) \right)$ iterations for strongly convex and Lipschitz functions. Our analysis shows that the performance of our algorithm depends on the algebraic connectivity of the graph, the minimum degree of the nodes, and the maximum degree of the nodes. Finally, we illustrated the performance of our algorithm with numerical examples. 

\section*{Appendix}
\subsection{\textup{\textbf{Proof of Lemma \autoref{lem:constraint}}}}
Consider the $k$-th coordinate of all $x_i$'s and form a $n \times 1$ vector $\mathbf{x}^k$. From $A\mathbf{x}=0$ and the fact that $A= P \otimes I$, we obtain that $P \mathbf{x}^k=0$. This shows that $\mathbf{x}^k \in \text{null}(P)$. Using Assumption \ref{assump:matrixA}, we have $\mathbf{x}^k \in \text{span}(\{\mathbf{1}\})$, which guarantees that all entries of $\mathbf{x}^k$ are equal. Similarly, for any $k=1, \dots, d$, the $k$-th entries of all $x_i$'s are equal. This leads to $x_i=x_j$ for all $i, j \in V$. 
\subsection{\textup{\textbf{Proof of Lemma \autoref{lem:purturbed}}}}

Using the first step of Algorithm \eqref{alg:ADMM}  for $i$, we can write $x_i(t+1)$ as  
\begin{align}\label{eq:new1}
& h_i (\mathbf{x}(t+1))+ \sum_{j \in N(i)} A_{ji} p_{j}(t) \nonumber \\
& + c  \sum_{j \in N(i)} A_{ji} ( y_j(t)+ A_{ji}x_i(t+1)- A_{ji} x_i(t) )=0,
\end{align}
where $h_i (\mathbf{x}(t+1))=f'_i(x_i(t+1))$ for differentiable functions and $h_i (\mathbf{x}(t+1)) \in \partial f_i(x_i(t+1))$ in general. 
We next use second and third steps of Algorithm \eqref{alg:ADMM} to write $p_i(t)$ and $y_i(t)$ in terms of $(\mathbf{x}(0), \dots, \mathbf{x}(t))$. 
Using the update for $j$, we have
\begin{align}\label{eq:new2}
& \sum_{j \in N(i)} p_j(t) A_{ji}= \sum_{j \in N(i)} A_{ji} \sum_{s=0}^t \frac{c}{d_i+1} \sum_{s=0}^t ([A]^j)' \mathbf{x}(s)\nonumber\\
& = c  \sum_{s=0}^t [A' D^{-1} A x(s)]_i. 
\end{align}
Moreover, we can write the term $\sum_{j \in N(i)} A_{ji} y_{j}(t)$ based on the sequence $(\mathbf{x}(0), \dots, \mathbf{x}(t))$ as follows
\begin{align}\label{eq:new3}
\sum_{j \in N(i)} A_{ji} y_{j}(t) &= \sum_{j \in N(i)} A_{ji} \frac{1}{d_j+1} ([A]^j)'\mathbf{x}(t)\nonumber\\
& = [A' D^{-1} A \mathbf{x}(t)]_i.
\end{align}
Substituting \eqref{eq:new2} and \eqref{eq:new3} in \eqref{eq:new1}, we can write the update of $x_{i}(t+1)$ in terms of the sequence $(\mathbf{x}(0), \dots, \mathbf{x}(t))$, which then can compactly be written as 
\begin{align*}
c M \mathbf{x}(t+1) & = - h(\mathbf{x}(t+1)) + c \left( M - A' D^{-1} A \right) \mathbf{x}(t) \\
& - c (A' D^{-1} A) \sum_{s=0}^t \mathbf{x}(s),
\end{align*}
where $h(\mathbf{x}(t+1)) = \nabla F(\mathbf{x}(t+1))$ if the functions are differentiable and $h(\mathbf{x}(t+1)) \in \partial F(\mathbf{x}(t+1))$ in general. Left multiplying by $\frac{1}{c}M^{-1}$, completes the proof. 

\subsection{\textup{\textbf{Proof of Proposition \autoref{pro:telescopic}}}}
We first show a lemma that we will use in the proof of this proposition. The following lemma shows the relation between the auxiliary sequence $\mathbf{r}(t)$ and the primal sequence $\mathbf{x}(t)$. 

\begin{lemma}\label{lem:recursion}
 Suppose Assumptions \ref{assump:matrixA} and   \ref{assump:nonempty-set} hold. The sequence $\{\mathbf{x}(t), \mathbf{r}(t)\}_{t=0}^{\infty}$ satisfies 
\begin{align*}
& (M- A' D^{-1} A) (\mathbf{x}(t+1)- \mathbf{x}(t)) \nonumber \\
& = -Q \mathbf{r}(t+1)- \frac{1}{c} h(\mathbf{x}(t+1)),
\end{align*}
for some $h(\mathbf{x}(t+1)) \in \partial F(\mathbf{x}(t+1))$. 
\end{lemma}
\begin{proof}
Using Lemma \ref{lem:purturbed} we have 
\begin{align*}
& M(\mathbf{x}(t+1)- \mathbf{x}(t)) = - \frac{1}{c} h(\mathbf{x}(t+1)) \\
& - (A' D^{-1} A) \mathbf{x}(t) - (A' D^{-1} A) \sum_{s=0}^{t} \mathbf{x}(s).
\end{align*}
We subtract $(A' D^{-1} A) \mathbf{x}(t+1)$ from both sides and rearrange the terms to obtain 
\begin{align*}
& (M- A' D^{-1} A) (\mathbf{x}(t+1)- \mathbf{x}(t))\\
& = - A' D^{-1} A \sum_{s=0}^{t+1} \mathbf{x}(s)- \frac{1}{c} h(\mathbf{x}(t+1)).
\end{align*}
Using $QQ=A'D^{-1} A$, yields 
\begin{align*}
& (M- A' D^{-1} A) (\mathbf{x}(t+1)- \mathbf{x}(t))\\
& = - Q \mathbf{r}(t+1) - \frac{1}{c} h(\mathbf{x}(t+1)).
\end{align*}
 
\end{proof}
\textbf{Back to the proof of Proposition \autoref{pro:telescopic}:} Using Lemma \ref{lem:recursion} and Lemma \ref{lem:strongconvexandlipschitz} part (c), we have that 
\begin{align*}
& \frac{2}{c} \left( F(\mathbf{x}(t+1))- F(\mathbf{x}^*) \right) + 2 \mathbf{r}' Q \mathbf{x}(t+1) \\
& \le \frac{2}{c} (\mathbf{x}(t+1)- \mathbf{x}^*)' h(\mathbf{x}(t+1)) + 2 \mathbf{r}' Q \mathbf{x}(t+1)\\
& = 2 (\mathbf{x}(t+1)- \mathbf{x}^*)' (  -Q (\mathbf{r}(t+1)- \mathbf{r})\\
& - (M- A' D^{-1} A) (\mathbf{x}(t+1)- \mathbf{x}(t))) \\
&= 2 (\mathbf{r}(t+1)- \mathbf{r}(t))' \left(  -\mathbf{r}(t+1)+ \mathbf{r} \right)  \\
& - 2 (\mathbf{x}(t+1)- \mathbf{x}^*)' (M- A' D^{-1} A) (\mathbf{x}(t+1)- \mathbf{x}(t))\\
& = \left( ||\mathbf{r}(t)- \mathbf{r}^*||_2^2 - ||\mathbf{r}(t+1)- \mathbf{r}||_2^2 - ||\mathbf{r}(t)- \mathbf{r}(t+1)||_2^2\right) \\
& + ( ||\mathbf{x}(t)- \mathbf{x}^*||_{(M- A' D^{-1} A)}^2 - ||\mathbf{x}(t+1)- \mathbf{x}^*||_{(M- A' D^{-1} A)}^2 \\
& - ||\mathbf{x}(t)- \mathbf{x}(t+1)||_{(M- A' D^{-1} A)}^2 )\\
& = ||\mathbf{q}(t)- \mathbf{q}^*||_G^2 - ||\mathbf{q}(t+1)- \mathbf{q}^*||_G^2 - ||\mathbf{q}(t)- \mathbf{q}(t+1)||_G^2.
\end{align*}

\subsection{\textup{\textbf{Proof of \autoref{thm:sublinearrate}}}}
Taking summation of the relation given in Proposition \ref{pro:telescopic} from $t=0$ up to $t=T$, we obtain that 
\begin{align*}
\sum_{t=0}^T F(\mathbf{x}(t))- F(\mathbf{x}^*) + c \mathbf{r}' Q \mathbf{x}(t) \le \frac{c}{2}||\mathbf{q}(0)-\mathbf{q}||_G^2.
\end{align*}
Using convexity of the functions and Jensen's inequality, we obtain  
\begin{align*}
F(\hat{\mathbf{x}}(T))- F(\mathbf{x}^*) + c \mathbf{r}' Q \hat{\mathbf{x}}(T) \le \frac{c}{2T} ||\mathbf{q}(0)-\mathbf{q}||_G^2.
\end{align*} 
Letting $\mathbf{r}=0$, yields  
\begin{align}\label{eq:optimalityproof1}
&   F(\hat{\mathbf{x}}(T)) - F(\mathbf{x}^*)  \nonumber\\
& \le  \frac{c}{2T} \left(||\mathbf{x}(0)-\mathbf{x}^*||_{M- A' D^{-1} A}^2+ ||\mathbf{r}(0)||_2^2 \right).
\end{align}

From saddle point inequality, we have   
\begin{align}\label{eq:optimalityproof2}
F(\mathbf{x}^*) \le F(\hat{\mathbf{x}}(T)) + c {\hat{\mathbf{r}}}' Q \hat{\mathbf{x}}(T),
\end{align}
which implies 
\begin{align}\label{eq:optimalityproof3}
F(\mathbf{x}^*) - F(\hat{\mathbf{x}}(T)) \le  c {\hat{\mathbf{r}}}' Q \hat{\mathbf{x}}(T).
\end{align}
Next, we will bound the term ${\hat{\mathbf{r}}}' Q \hat{\mathbf{x}}(T)$. We add the term $ c {\hat{\mathbf{r}}}' Q \hat{\mathbf{x}}(T)$ to both sides of \eqref{eq:optimalityproof2} to obtain 
\begin{align}\label{eq:optimalityproof4}
c {\hat{\mathbf{r}}}' Q \hat{\mathbf{x}}(T) \le F(\hat{\mathbf{x}}(T)) - F(\mathbf{x}^*)  + 2 c {\hat{\mathbf{r}}}' Q \hat{\mathbf{x}}(T). 
\end{align}
Again, using Proposition \ref{pro:telescopic} to bound the right-hand side of \eqref{eq:optimalityproof4}, we obtain 
\begin{align}\label{eq:optimalityproof5}
& c {\mathbf{r}^*}' Q \hat{\mathbf{x}}(T) \le \frac{c}{2T} \left(||\mathbf{x}(0)-\mathbf{x}^*||_{M- A' D^{-1} A}^2+ ||\mathbf{r}(0)- 2 \hat{\mathbf{r}}||_2^2 \right).
\end{align}
Using \eqref{eq:optimalityproof5} to bound the right-hand side of \eqref{eq:optimalityproof3}, and then 
combining the result with  \eqref{eq:optimalityproof2}, we obtain 
\begin{align*}
& | F(\hat{\mathbf{x}}(T)) - F(\mathbf{x}^*)|  \le \frac{c}{2T} \left(||\mathbf{x}(0)-\mathbf{x}^*||_{M- A' D^{-1} A}^2 \right) \\
& + \frac{c}{2T} \left( \max\{ ||\mathbf{r}(0)- 2 \hat{\mathbf{r}}||_2^2, ||\mathbf{r}(0)||_2^2 \} \right).
\end{align*}
We next bound the feasibility violation. Using Proposition \ref{pro:telescopic} with $\mathbf{r}=\hat{\mathbf{r}}+ \frac{Q \hat{\mathbf{x}}(T)}{||Q \hat{\mathbf{x}}(T)||}$, we have 
\begin{align*}
&   F(\hat{\mathbf{x}}(T)) - F(\mathbf{x}^*) + c {\mathbf{r}^*}' Q \hat{\mathbf{x}}(T) + c ||Q \hat{\mathbf{x}}(T)|| \nonumber\\
& \le  \frac{c}{2T} \left(||\mathbf{x}(0)-\mathbf{x}^*||_{M- A' D^{-1} A}^2 \right) \\
& + \frac{c}{2T} \left( ||\mathbf{r}(0)-  \hat{\mathbf{r}} - \frac{Q \hat{\mathbf{x}}(T)}{||Q \hat{\mathbf{x}}(T)||} ||_2^2 \right).
\end{align*}
Since $(\mathbf{x}^*, \hat{\mathbf{r}})$ is a primal-dual optimal solution, using saddle point inequality, we have that 
\begin{align*}
F(\hat{\mathbf{x}}(T)) - F(\mathbf{x}^*) + c {\hat{\mathbf{r}}}' Q \hat{\mathbf{x}}(T) \ge 0.
\end{align*}
Combining the two previous relations, we obtain 
\begin{align*}
&   ||Q \hat{\mathbf{x}}(T)||_2  \le  \frac{1}{2T} \left(||\mathbf{x}(0)-\mathbf{x}^*||_{M- A' D^{-1} A}^2 \right) \\
& + \frac{1}{2T} \left( ||\mathbf{r}(0)-  \hat{\mathbf{r}}- \frac{Q \hat{\mathbf{x}}(T)}{||Q \hat{\mathbf{x}}(T)||} ||_2^2 \right).
\end{align*}
Since 
\begin{align*}
||\mathbf{r}(0)-  \hat{\mathbf{r}}- \frac{Q \hat{\mathbf{x}}(T)}{||Q \hat{\mathbf{x}}(T)||} ||_2^2  \le 2 ||\mathbf{r}(0)-  \hat{\mathbf{r}}||_2^2 + 2,
\end{align*} 
we can further bound $||Q \hat{\mathbf{x}}(T)||$ as  
\begin{align*}
&   ||Q \hat{\mathbf{x}}(T)||_2  \le  \frac{1}{2T} \left(||\mathbf{x}(0)-\mathbf{x}^*||_{M- A' D^{-1} A}^2 \right) \\
& + \frac{1}{2T} \left( 2||\mathbf{r}(0)-  \hat{\mathbf{r}}||_2^2+2 \right).
\end{align*}

\subsection{\textup{\textbf{Proof of \autoref{thm:sublinearnetworkdependence}}}}
We first show a lemma that bounds the norm of dual optimal solution of \eqref{eq:optlinearconstraintwithQ}.  
\begin{lemma}
{Let $\mathbf{x}^*$ be an optimal solution for problem \eqref{eq:optlinearconstraint}}. There exists an  optimal dual solution $\tilde{\mathbf{r}}$ for problem 
\begin{align}\label{eq:optlinearconstraintwithQ}
\min_{cQ \mathbf{x}=0} F(\mathbf{x}),
\end{align}
 that satisfies 
\begin{align*}
||\tilde{\mathbf{r}}||_2^2 \le \frac{U^2}{c^2 \tilde{\lambda}_m},
\end{align*}
where $U$ is a bound on the subgradients of the function $F$ at $\mathbf{x}^*$, i.e., $\|\mathbf{v}\|\le U$ for all $\mathbf{v}\in \partial F(\mathbf{x}^*)$, and 
$\tilde{\lambda}_m$ is the smallest non-zero eigen-value of $A' D^{-1} A$. 
\end{lemma}\label{lem:boundpstar}


\begin{proof}
There exists an optimal primal-dual solution for problem \eqref{eq:optlinearconstraintwithQ} such that $(\mathbf{x}^*, \hat{\mathbf{r}})$ is a saddle point of the Lagrangian function, i.e., for any ${\bf x}\in \mathbb{R}^n$,
\begin{align}\label{eq:leboundstarprooftemp}
F(\mathbf{x}^*) - F(\mathbf{x}) \le  {c \hat{\mathbf{r}}}' Q (\mathbf{x} - \mathbf{x}^*).
\end{align}
Note that $(\mathbf{x}^*, \hat{\mathbf{r}})$ satisfies saddle point inequality if and only if it satisfies the inequality given in  \eqref{eq:leboundstarprooftemp}. Equation \eqref{eq:leboundstarprooftemp} shows that $c {\hat{\mathbf{r}}}'Q \in \partial F(\mathbf{x}^*)$. Let $c {\hat{\mathbf{r}}}'Q=\mathbf{v}' \in \partial F(\mathbf{x}^*)$.  We will use this $\hat{\mathbf{r}}$ to construct $\tilde{\mathbf{r}}$ such that $c \tilde{\mathbf{r}}'Q = \mathbf{v}'$ and hence we would have 
\begin{align*}
F(\mathbf{x}^*) - F(\mathbf{x}) \le  c {\tilde{\mathbf{r}}}' Q (\mathbf{x} - \mathbf{x}^*),
\end{align*}
meaning $(\mathbf{x}^*,\tilde{\mathbf{r}})$ satisfies the saddle point inequality. This shows that $(\mathbf{x}^*,\tilde{\mathbf{r}})$ is an optimal primal-dual solution (see section 6 of \cite{bertsekas2003convex}). Moreover, we choose $\tilde{\mathbf{r}}$ to satisfy the statement of lemma. 

Let $Q= \sum_{i=1}^r \mathbf{u}_i \sigma_i \mathbf{v}'_i$ be the singular value decomposition of $Q$, where $\text{rank}(Q)=r$.  Since $c {\hat{\mathbf{r}}}'Q=\mathbf{v}'$, $\mathbf{v}$  belongs to the span of $\{c \mathbf{v}_1, \dots, c \mathbf{v}_r\}$ and can be written as $\mathbf{v}= c \sum_{i=1}^r c_i \mathbf{v}_i$ for some coefficients $c_i$'s . Let $\tilde{\mathbf{r}}= \sum_{i=1}^r \frac{c_i}{\sigma_i} \mathbf{u}_i$. By this choice of $\tilde{\mathbf{r}}$   we have $ c\tilde{\mathbf{r}}' Q = c\sum_{i=1}^r c_i \mathbf{v}_i= \mathbf{v}'$. This choice also yields 
 
\begin{align*}
&  ||\tilde{\mathbf{r}}||^2=  \sum_{i=1}^r \frac{c_i^2}{\sigma_i^2} \le  \sum_{i=1}^r \frac{c_i^2}{\tilde{\lambda}_m}=  \frac{1}{\sigma_{\text{min}}^2} \sum_{i=1}^r c_i^2 \\
& =  \frac{1}{c^2 \tilde{\lambda}_m} ||\mathbf{v}||^2 \le  \frac{1}{c^2 \tilde{\lambda}_m} U^2,
\end{align*} 
where we used the bound on the subgradient to obtain the last inequality. Since ${\tilde{\mathbf{r}}}' Q=  \mathbf{v}' \in \partial F(\mathbf{x}^*)$, $(\mathbf{x}^*, \tilde{\mathbf{r}})$ satisfies the saddle point inequality.


\end{proof}
Next, we use this lemma to analyze the network effect. 
Using Theorem \ref{thm:sublinearrate} with zero initial condition, we have 
\begin{align*}
 | F(\hat{\mathbf{x}}(T)) - F(\mathbf{x}^*)| \le & \frac{c}{2T} ||\mathbf{x}^*||_{M- A'D^{-1}A}^2 +  \frac{2}{c T} \frac{U^2}{\tilde{\lambda}_m}, 
\end{align*}
and 
\begin{align*}
   ||Q \hat{\mathbf{x}}(T)|| \le  \frac{1}{2T} ||\mathbf{x}^*||_{M- A' D^{-1} A}^2   + \frac{1}{2T} \left( 2+ 2\frac{U^2}{c^2 \tilde{\lambda}_m} \right).
\end{align*}
Using $||\mathbf{x}^*||_{M- A' D^{-1} A}^2  \le \lambda_M ||\mathbf{x}^*||_2^2$ completes the proof. 

\subsection{\textup{\textbf{Proof of Lemma \autoref{lem:strongconvexandlipschitz}}}}

For any $i$, since $f_i(x)- \frac{\nu_{f_i}}{2} ||x||^2$ is convex, $\nabla ( f_i(x)- \frac{\nu_{f_i}}{2} ||x||^2)$ is monotone and we have 
\begin{align*}
\left< \nabla f_i(x)- \nabla f_i(y), x-y \right> - \nu_{f_i} ||x-y||^2 \ge 0.
\end{align*}
Since $\nu \le \nu_{f_i}$, we obtain
\begin{align*}
\left< \nabla f_i(x)- \nabla f_i(y), x-y \right> - \nu ||x-y||^2 \ge 0,
\end{align*}
which results in 
\begin{align*}
(\nabla F(\mathbf{x})- \nabla F(\mathbf{y}))' (\mathbf{x}- \mathbf{y}) \ge \nu ||\mathbf{x}- \mathbf{y}||_2^2,
\end{align*}
this completes the proof of part (a).  We now prove part (b). For any $x, y \in \mathbb{R}^n$, we have 
\begin{align*}
f_i(y) &= f_i(x) + \int_0^1 \left< {\nabla} f_i (x+ \tau (y -x)), y-x\right> d\tau \\
& = f_i(x) + \left< {\nabla} f_i(x), y-x \right>\\
& + \int_0^1 \left< {\nabla} f_i ((1-\tau)x+ \tau y ) - {\nabla} f_i(x), y-x\right> d\tau\\
& \le f_i(x) + \left< {\nabla} f(x), y-x \right>\\
& + \int_0^1 || {\nabla} f_i ((1-\tau)x+ \tau y)- {\nabla} f_i(x) || || y-x|| d\tau\\
& \le f_i(x) + \left< {\nabla} f_i(x), y-x \right> + L_{f_i} ||y-x||^2 \int_0^1 \tau d\tau\\
& = f_i(x) + \left< {\nabla} f_i(x), y-x \right>+ \frac{L_{f_i}}{2} ||y-x||^2.
\end{align*}
Let $\phi_x(y)= f_i(y)-\left< {\nabla} f_i(x), y\right>$. Note that $\phi_x(y)$ has Lipschitz gradient with parameter $L_{f_i}$. Moreover, we have that $\min_{y} \phi_x(y)= \phi_x(x)$, since 
\begin{align*}
{\nabla} \phi_x(y)= {\nabla} f_i(y)- {\nabla} f_i(x)
\end{align*}
is zero for $y=x$. Therefore, using the previous relation, we have that 
\begin{align*}
& \phi_x(y) - \phi_x(x) = f_i(y)- f_i(x) - \left< \nabla  f_i(x), y-x\right> \\
& \ge \frac{1}{2 L_{f_i}} ||\nabla \phi_x(y)||= \frac{1}{2 L_{f_i}} ||\nabla f_i(y)- \nabla f_i(x)||^2.
\end{align*}
 Using the previous relation , we have 
\begin{align*}
f_i(y)- f_i(x) - \left< \nabla  f_i(x), y-x\right> \ge \frac{1}{2 L_{f_i}} ||\nabla f_i(y)- \nabla f_i(x)||^2.
\end{align*}
We also have 
\begin{align*}
f_i(x)- f_i(y) - \left< \nabla  f_i(y), x-y\right> \ge  \frac{1}{2 L_{f_i}} ||\nabla f_i(y)- \nabla f_i(x)||^2.
\end{align*}
We add the two preceding relations to obtain 
\begin{align*}
\left< {\nabla} f_i(x) - {\nabla} f_i(y), x-y \right> \ge \frac{1}{L_{f_i}} ||{\nabla} f_i(y)- {\nabla} f_i(x)||^2,
\end{align*}
for any $i=1, \dots, n$. Combining this relation for all $i$'s completes the proof. 
\\Finally, we prove part (c). Let $h=(h_1', \dots, h_n')'$. By definition of subgradient, for any $i \in V$ we have 
\begin{align*}
h_i(x_i(t+1))' (x_i- y_i) \ge f_i(x_i)- f_i(y_i) .
\end{align*}
Taking summation of this inequality for all $i=1, \dots, n$ shows that 
\begin{align*}
h(\mathbf{x})' (\mathbf{x}- \mathbf{y}) \ge F(\mathbf{x})- F(\mathbf{y}). 
\end{align*}

\subsection{\textup{\textbf{Proof of  \autoref{thm:linearconvergence}}}}
We first show two lemmas that we use in the proof of this theorem. The first lemma shows that both matrices $M-A'D^{-1}A$ and $A'D^{-1}A$ are positive semidefinite and the second lemma shows a relation between the sequences $\mathbf{q}(t)$, $\mathbf{r}(t)$, and $\mathbf{x}(t)$ same as the one shown in Lemma \autoref{lem:recursion}. 
\begin{lemma}\label{lem:PSD}
The matrices $M-A'D^{-1}A$ and $A'D^{-1}A$ are positive semidefinite.
\end{lemma}
\begin{proof}
Both matrices are clearly symmetric. We first show $M- A' D^{-1} A$ is positive semidefinite. By definition, we have 
\begin{align*}
[M- A' D^{-1} A]_{ii}= M_{ii} - \sum_l A_{li} A_{li} \frac{1}{d_l +1}= \sum_l A_{li}^2  \frac{d_l}{d_l +1}.
\end{align*}
We also have 
\begin{align*}
[M- A' D^{-1} A]_{ij}= - \sum_l A_{li} A_{lj} \frac{1}{d_l+1} . 
\end{align*}
Therefore, 
\begin{align*}
& \sum_{j\neq i} \left.|[M- A' D^{-1} A]_{ij} \right.|= \sum_{j \neq i} \left.| \sum_l A_{li} A_{lj} \frac{1}{d_l+1} \right.| \\
& \le  \sum_l \left.| A_{li} \right.| \frac{1}{d_l+1} \left.| \sum_{j \neq i} A_{lj} \right.| = \sum_l A_{li}^2 \frac{1}{d_l+1},
\end{align*}
where we used the fact that $\sum_{j=1}^n A_{lj}=0$, for any $l$. Therefore, by Greshgorin Circle Theorem, for any eigen value $\mu$ of $M- A' D^{-1} A$, for some $i$ we have
\begin{align*}
|\mu- [M- A' D^{-1} A]_{ii} | \le \sum_{j\neq i} |[M- A' D^{-1} A]_{ij}|,
\end{align*}
which leads to 
\begin{align*}
& \mu \ge [M- A' D^{-1} A]_{ii} - \sum_{j\neq i} \left.|[M- A' D^{-1} A]_{ij} \right.| \\
& \ge \sum_l A_{li}^2 \left( \frac{d_l -1}{d_l+1} \right) \ge 0,
\end{align*}
where we used the fact that $d_l \ge 1$ that evidently holds. We next show that $A'D^{-1}A$ is positive semidefinite. We have 
\begin{align*}
[A'D^{-1}A]_{ii} = \sum_l A_{li}^2 \frac{1}{d_l +1}
\end{align*} 
We also have 
\begin{align*}
& \sum_{j\neq i} \left.| [A' D^{-1} A]_{ij} \right.|= \sum_{j \neq i} \left.| \sum_l A_{li} A_{lj} \frac{1}{d_l+1} \right.| \\
& \le \sum_l \left.| A_{li} \right.| \frac{1}{d_l+1} \left.| \sum_{j \neq i} A_{lj} \right.| = \sum_l A_{li}^2 \frac{1}{d_l+1}.
\end{align*}
Since $[A'D^{-1}A]_{ii} \ge \sum_{j\neq i} \left.| [A' D^{-1} A]_{ij} \right.|$, similarly, by Greshgorin Circle Theorem, the matrix $A' D^{-1} A$ is positive semidefinite. 
\end{proof}
\begin{lemma}\label{lem:recursion}
 Suppose Assumptions \ref{assump:matrixA} and   \ref{assump:nonempty-set} hold. For differentiable functions, the sequence $\{\mathbf{x}(t), \mathbf{r}(t)\}_{t=0}^{\infty}$ satisfies 
\begin{align*}
& (M- A' D^{-1} A) (\mathbf{x}(t+1)- \mathbf{x}(t)) \nonumber \\
& = -Q (\mathbf{r}(t+1)- \mathbf{r}^*)- \frac{1}{c} (\nabla F(\mathbf{x}(t+1))- \nabla F(\mathbf{x}^*)),
\end{align*}
for some $\mathbf{r}^*$ that satisfies $Q \mathbf{r}^* + \frac{1}{c} \nabla F(\mathbf{x}^*)=0$. Moreover, $\mathbf{r}^*$ belongs to the column span of $Q$. 
\end{lemma}
\begin{proof}
Using Lemma \ref{lem:purturbed} for differentiable functions we have 
\begin{align*}
& M(\mathbf{x}(t+1)- \mathbf{x}(t)) = - \frac{1}{c} \nabla F(\mathbf{x}(t+1)) \\
& - (A' D^{-1} A) \mathbf{x}(t) - (A' D^{-1} A) \sum_{s=0}^{t} \mathbf{x}(s).
\end{align*}
We subtract $(A' D^{-1} A) \mathbf{x}(t+1)$ from both sides and rearrange the terms to obtain 
\begin{align*}
& (M- A' D^{-1} A) (\mathbf{x}(t+1)- \mathbf{x}(t))\\
& = - A' D^{-1} A \sum_{s=0}^{t+1} \mathbf{x}(s)- \frac{1}{c} \nabla F(\mathbf{x}(t+1)).
\end{align*}
Using $QQ=A'D^{-1} A$, yields 
\begin{align*}
& (M- A' D^{-1} A) (\mathbf{x}(t+1)- \mathbf{x}(t))\\
& = - Q \mathbf{r}(t+1) - \frac{1}{c} \nabla F(\mathbf{x}(t+1)).
\end{align*}
 We next show there exist $\mathbf{r}^*$ such that 
$ Q \mathbf{r}^* + \frac{1}{c} \nabla F(\mathbf{x}^*)=0$.
First note that both column space (range) and null space of $Q$ and $A' D^{-1} A$ are the same. Since $\text{span}(Q)\oplus \text{null}(Q)= \mathbb{R}^n$, we have $\nabla F(\mathbf{x}^*) \in  \text{span}(Q)\oplus \text{null}(Q)= \text{span}(Q)\oplus \text{span}\{\mathbf{1}\}$ as $\text{null}(Q)=\text{span}\{\mathbf{1}\}$. Since $\mathbf{1}' \nabla F(\mathbf{x}^*)=0$, we can write $\nabla F(\mathbf{x}^*)$ as a linear combination of column vectors of $Q$. Therefore, there exist $\mathbf{r}$ such that $\frac{1}{c} \nabla F(\mathbf{x}^*)= - Q \mathbf{r}$. Let $\mathbf{r}^*= \text{Proj}_Q \mathbf{r}$ to obtain $Q \mathbf{r}= Q \mathbf{r}^*$ where $\mathbf{r}^*$ lies in the column space of $Q$. Part (b) simply follows from the same lines of argument. 
\end{proof}
\textbf{Back to the proof of \autoref{thm:linearconvergence}:} Note that since $M- A' D^{-1} A$ is positive semidefinite, 
\begin{align*}
& \left\langle . , . \right\rangle : \mathbb{R}^{2n} \times \mathbb{R}^{2n} \mapsto \mathbb{R} \\
& \left\langle \mathbf{q}_1, \mathbf{q}_2 \right\rangle= \mathbf{q}'_1 G \mathbf{q}_2,
\end{align*} 
where \[G= \begin{pmatrix}
  I & 0  \\
  0 & M- A'D^{-1}A 
 \end{pmatrix}\] is a semi-inner product.\footnote{This means it satisfies conjugate symmetry, linearity and semipositive-definiteness (instead of positive-definiteness).}  We first show that for a $\delta$ given by the statement of theorem, we have 
\begin{align}\label{eq:pftemp0}
||\mathbf{q}(t+1)- \mathbf{q}^*||_G^2 \le \left( \frac{1}{1+ \delta} \right) ||\mathbf{q}(t)- \mathbf{q}^*||_G^2. 
\end{align}
Using Lemma \eqref{lem:strongconvexandlipschitz} and Lemma \eqref{lem:recursion}, we have  
\begin{align}\label{eq:pftemp1}
& \frac{2}{c}\nu ||\mathbf{x}(t+1)- \mathbf{x}^*||_2^2 \nonumber \\
& \le \frac{2}{c} (\mathbf{x}(t+1)- \mathbf{x}^*)' (\nabla F(\mathbf{x}(t+1))- \nabla F(\mathbf{x}^*)) \nonumber \\
& = 2 (\mathbf{x}(t+1)- \mathbf{x}^*)' (Q (\mathbf{r}^*- \mathbf{r}(t+1))) \nonumber \\
& + 2 (\mathbf{x}(t+1)- \mathbf{x}^*)' (M- A' D^{-1} A) (\mathbf{x}(t)- \mathbf{x}(t+1)) \nonumber\\
& = 2 (\mathbf{r}(t+1)- \mathbf{r}(t))' (\mathbf{r}^*- \mathbf{r}(t+1)) \nonumber \\
& + 2 (\mathbf{x}(t+1)- \mathbf{x}(t))' (M- A' D^{-1} A) (\mathbf{x}^*- \mathbf{x}(t+1))  \nonumber\\
&=  2 (\mathbf{q}(t+1)- \mathbf{q}(t))' G (\mathbf{q}^*- \mathbf{r}(t+1))   \nonumber\\
& =  ||\mathbf{q}(t)- \mathbf{q}^*||_G^2 - ||\mathbf{q}(t+1)- \mathbf{q}^*||_G^2 - ||\mathbf{q}(t)- \mathbf{q}(t+1)||_G^2 .
\end{align}
Again, using Lemma \eqref{lem:strongconvexandlipschitz} and Lemma \eqref{lem:recursion}, we have 
\begin{align}\label{eq:pftemp1.5}
& \frac{2}{c}\frac{1}{L} ||\nabla F(\mathbf{x}(t+1))- \nabla F (\mathbf{x}^*)||_2^2 \nonumber \\
& \le  ||\mathbf{q}(t)- \mathbf{q}^*||_G^2 - ||\mathbf{q}(t+1)- \mathbf{q}^*||_G^2 - ||\mathbf{q}(t)- \mathbf{q}(t+1)||_G^2 .
\end{align}
Using \eqref{eq:pftemp1} and \eqref{eq:pftemp1.5}, for any $\beta \in (0,1)$, we have 
\begin{align}\label{eq:pftemp1.6}
& \beta \frac{2}{c}\nu ||\mathbf{x}(t+1)- \mathbf{x}^*||_2^2  \nonumber \\
& + (1- \beta) \frac{2}{c}\frac{1}{L} ||\nabla F(\mathbf{x}(t+1))- \nabla F (\mathbf{x}^*)||_2^2  \\
\le &   ||\mathbf{q}(t)- \mathbf{q}^*||_G^2 - ||\mathbf{q}(t+1)- \mathbf{q}^*||_G^2 - ||\mathbf{q}(t)- \mathbf{q}(t+1)||_G^2 .
\end{align}
This yields to 
\begin{align}\label{eq:pftemp2}
& ||\mathbf{q}(t)- \mathbf{q}^*||_G^2 - ||\mathbf{q}(t+1)- \mathbf{q}^*||_G^2 \nonumber \\
& \ge ||\mathbf{q}(t)- \mathbf{q}(t+1)||_G^2 + \beta \frac{2}{c}\nu ||\mathbf{x}(t+1)- \mathbf{x}^*||_2^2  \nonumber \\
& + (1- \beta) \frac{2}{c}\frac{1}{L} ||\nabla F(\mathbf{x}(t+1))- \nabla F (\mathbf{x}^*)||_2^2 . 
\end{align}
Comparing this relation with \eqref{eq:pftemp0}, it remains to show  
\begin{align*}
& ||\mathbf{q}(t)- \mathbf{q}(t+1)||_G^2 + \beta \frac{2}{c}\nu ||\mathbf{x}(t+1)- \mathbf{x}^*||_2^2 \\
&  + (1- \beta) \frac{2}{c}\frac{1}{L} ||\nabla F(\mathbf{x}(t+1))- \nabla F (\mathbf{x}^*)||_2^2 \\
& \ge \delta ||\mathbf{q}(t+1)- \mathbf{q}^*||_G^2,
\end{align*}
which is equivalent to 
\begin{align}
& ||\mathbf{q}(t)- \mathbf{q}(t+1)||_{G}^2 + ||\mathbf{x}(t+1)- \mathbf{x}^*||_{\frac{2 \nu t}{c}I - \delta (M- A' D^{-1} A)}^2 \nonumber \\
& + (1- \beta) \frac{2}{c}\frac{1}{L} ||\nabla F(\mathbf{x}(t+1))- \nabla F (\mathbf{x}^*)||_2^2 \nonumber\\
& \ge \delta ||\mathbf{r}(t+1)- \mathbf{r}^*||_2^2.
\end{align}
Using Lemma \eqref{lem:PSD}, in order to show this inequality it suffices to show 
 \begin{align}\label{eq:pftemp2.2}
&  ||\mathbf{x}(t+1)- \mathbf{x}^*||_{\frac{2 \nu t}{c}I - \delta (M- A' D^{-1} A)}^2  \nonumber \\
& + (1- \beta) \frac{2}{c}\frac{1}{L} ||\nabla F(\mathbf{x}(t+1))- \nabla F (\mathbf{x}^*)||_2^2 \nonumber \\
& \ge \delta ||\mathbf{r}(t+1)- \mathbf{r}^*||_2^2.
\end{align}
Since both $\mathbf{r}(t+1)$ and $\mathbf{r}^*$ are orthogonal to $\mathbf{1}$ and $\text{null}(Q)= \text{span}(\{\mathbf{1}\})$, using Lemma \eqref{lem:recursion}, we obtain
\begin{align}\label{eq:temppf2.45}
& \delta ||(\mathbf{r}(t+1)- \mathbf{r}^*)||_2^2  \le \frac{\delta}{ \tilde{\lambda}_m} ||Q (\mathbf{r}(t+1)- \mathbf{r}^*)||_2^2 \nonumber \\
& \le \frac{\delta}{ \tilde{\lambda}_m} ||(M - A' D^{-1} A)(\mathbf{x}(t+1)- \mathbf{x}^*) \nonumber\\
& - \frac{1}{c} (\nabla F(\mathbf{x}(t+1))- \nabla F (\mathbf{x}^*))||_2^2 \nonumber \\
& \le  \frac{2 \delta}{ \tilde{\lambda}_m} ||(M - A' D^{-1} A)(\mathbf{x}(t+1)- \mathbf{x}^*)||_2^2  \nonumber\\
& + \frac{2 \delta}{ \tilde{\lambda}_m} || (\nabla F(\mathbf{x}(t+1))- \nabla F (\mathbf{x}^*))||_2^2  \nonumber \\
& \le  \frac{2 \delta \lambda_M}{ \tilde{\lambda}_m} ||\mathbf{x}(t+1)- \mathbf{x}^*||_{M- A'D^{-1} A}^2  \nonumber\\
& + \frac{2 \delta}{ \tilde{\lambda}_m} || (\nabla F(\mathbf{x}(t+1))- \nabla F (\mathbf{x}^*))||_2^2
\end{align}
Comparing \eqref{eq:temppf2.45} and \eqref{eq:pftemp2.2}, it suffices to have 
\begin{align*}
\delta \le \min \left \{ \frac{2 \beta \nu }{c \lambda_M (1+ \frac{2}{\tilde{\lambda}_m})}, \frac{(1-\beta) c \tilde{\lambda}_m }{L} \right \}. 
\end{align*}
This shows that \eqref{eq:pftemp0} holds. Using \eqref{eq:pftemp0} along with \eqref{eq:pftemp1} completes the proof. 
\subsection{\textup{\textbf{Proof of  \autoref{thm:optimalstepsize}}}}

The largest possible $\delta$ that satisfies the constraint given in Theorem 1 by maximizing over $\beta \in (0,1)$ is the solution of 
\begin{align}\label{eq:cortemp1}
\frac{2 \beta \nu }{c \lambda_M (1+ \frac{2}{\tilde{\lambda}_m})}= \frac{(1- \beta) c \tilde{\lambda}_m }{L},
\end{align}
which is $\beta^*= \frac{c^2 \lambda_M (2+ \tilde{\lambda}_m)}{2 \nu L + c^2 \lambda_M (2+ \tilde{\lambda}_m)}$. This in turn shows that the maximum $\delta$ is equal to $\delta= \frac{2\beta^* \nu }{c \lambda_M (1+ \frac{2}{\tilde{\lambda}_m})}$. We now maximize $\delta$ over choices of $c$, leading to 
\begin{align*}
\delta^*= \frac{1}{2} \sqrt{\frac{2 \tilde{\lambda}_m^2}{ \lambda_M (2 + \tilde{\lambda_m})}\frac{1}{\kappa_f}}. 
\end{align*}

\subsection{\textup{\textbf{Proof of Proposition \autoref{pro:sublinearexplicitnetworkdep}}}}

The bound provided in \autoref{thm:sublinearnetworkdependence} depends on $\tilde{\lambda}_m$. We have that 
\begin{align*}
\tilde{\lambda}_m \ge \frac{1}{d_{\text{max}}+1} a(G)^2,
\end{align*}
where $a(G)$ is the algebraic connectivity of the graph which is the smallest non-zero eigenvalue of the Laplacian matrix. Moreover, we have that 
\begin{align*}
||M- A'D^{-1} A||_2 \le d_{\text{max}} (d_{\text{max}+1})+ \frac{4d_{\text{max}}^2}{d_{\text{min}}+1}.
\end{align*}
Plugging these two bounds in \autoref{thm:sublinearnetworkdependence} an using the fact that $d_{\text{min}} \ge 1$ completes the proof. 
\subsection{\textup{\textbf{Proof of Proposition \autoref{pro:linearexplicitnetworkdep}}}}
Using \autoref{thm:optimalstepsize}, for large enough $\kappa_f$(small enough $\delta^*$) we have $\frac{1}{\log(1+ \delta^*)} \ge \frac{1}{\delta^*}$, in order to have $||\mathbf{x}(t)- \mathbf{x}^*||_2 \le \epsilon$, we need to have 
\begin{align}\label{eq:prooflinearnetworkdep}
& t  \ge \frac{1}{\log \frac{1}{\rho}} \left( 2 \log (\frac{1}{\epsilon}) - \log (\frac{c^*}{2 \nu} ||\mathbf{q}(0)- \mathbf{q}^* ||_G^2) \right)  \nonumber \\
& \ge \sqrt{\kappa_f}\frac{2\sqrt{\lambda_M(2+ \tilde{\lambda}_m)}}{\sqrt{2} \tilde{\lambda}_m} \left( 2 \log (\frac{1}{\epsilon}) - \log (\frac{c^*}{2 \nu} ||\mathbf{q}(0)- \mathbf{q}^* ||_G^2) \right).
\end{align}
This shows that $O\left( \sqrt{\kappa_f} \frac{\sqrt{\lambda_M(2+ \tilde{\lambda}_m)}}{ \tilde{\lambda}_m} \log (\frac{1}{\epsilon}) \right)$ iterations suffice to have $||\mathbf{x}(t)- \mathbf{x}^*||_2 \le \epsilon$. This bound depends on $\tilde{\lambda}_m $ and $\lambda_M$. We have the following bounds 
\begin{align*}
\frac{1}{d_{\text{min}}+1} a(G)^2  \ge \tilde{\lambda}_m \ge \frac{1}{d_{\text{max}}+1} a(G)^2. 
\end{align*}
and 
\begin{align*}
\lambda_M \le  d_{\text{max}} (d_{\text{max}+1})+ \frac{\lambda_{\text{max}}(A)^2}{d_{\text{min}}+1}.
\end{align*}
Using these two bounds along with $\lambda_{\text{max}}(A) \le 2 d_{\text{max}}$, we obtain 
\begin{align*}
\frac{\lambda_M (2+ \tilde{\lambda}_m)}{\tilde{\lambda}_m^2} \le 16 \frac{d^4_{\text{max}}}{d_{\text{min}} a^2(G)}. 
\end{align*}
Plugging this bound into \eqref{eq:prooflinearnetworkdep} completes the proof. 

\bibliographystyle{./biblio/IEEEtran}
\bibliography{./biblio/IEEEabrv,references}

\begin{thebibliography}{10}
\providecommand{\url}[1]{#1}
\csname url@samestyle\endcsname
\providecommand{\newblock}{\relax}
\providecommand{\bibinfo}[2]{#2}
\providecommand{\BIBentrySTDinterwordspacing}{\spaceskip=0pt\relax}
\providecommand{\BIBentryALTinterwordstretchfactor}{4}
\providecommand{\BIBentryALTinterwordspacing}{\spaceskip=\fontdimen2\font plus
\BIBentryALTinterwordstretchfactor\fontdimen3\font minus
  \fontdimen4\font\relax}
\providecommand{\BIBforeignlanguage}[2]{{%
\expandafter\ifx\csname l@#1\endcsname\relax
\typeout{** WARNING: IEEEtran.bst: No hyphenation pattern has been}%
\typeout{** loaded for the language `#1'. Using the pattern for}%
\typeout{** the default language instead.}%
\else
\language=\csname l@#1\endcsname
\fi
#2}}
\providecommand{\BIBdecl}{\relax}
\BIBdecl

\bibitem{xiao2007distributed}
L.~Xiao, S.~Boyd, and S.-J. Kim, ``Distributed average consensus with
  least-mean-square deviation,'' \emph{Journal of Parallel and Distributed
  Computing}, 2007.

\bibitem{dekel2012optimal}
O.~Dekel, R.~Gilad-Bachrach, O.~Shamir, and L.~Xiao, ``Optimal distributed
  online prediction using mini-batches,'' \emph{The Journal of Machine Learning
  Research, 2012}.

\bibitem{recht2011hogwild}
B.~Recht, C.~Re, S.~Wright, and F.~Niu, ``Hogwild: A lock-free approach to
  parallelizing stochastic gradient descent,'' in \emph{Advances in Neural
  Information Processing Systems, 2011}.

\bibitem{agarwal2011distributed}
A.~Agarwal and J.~C. Duchi, ``Distributed delayed stochastic optimization,'' in
  \emph{Advances in Neural Information Processing Systems, 2011}.

\bibitem{duchi2014optimality}
J.~C. Duchi, M.~I. Jordan, M.~J. Wainwright, and Y.~Zhang, ``Optimality
  guarantees for distributed statistical estimation,'' \emph{arXiv preprint
  arXiv:1405.0782}, 2014.

\bibitem{duarte2004vehicle}
M.~F. Duarte and Y.~H. Hu, ``Vehicle classification in distributed sensor
  networks,'' \emph{Journal of Parallel and Distributed Computing, 2014}.

\bibitem{rabbat2004distributed}
M.~Rabbat and R.~Nowak, ``Distributed optimization in sensor networks,'' in
  \emph{International symposium on Information processing in sensor networks,
  2004}.

\bibitem{low1999optimization}
S.~H. Low and D.~E. Lapsley, ``Optimization flow control?i: basic algorithm and
  convergence,'' \emph{IEEE/ACM Transactions on Networking (TON), 1999}.

\bibitem{jadbabaie2003coordination}
A.~Jadbabaie, J.~Lin, and A.~S. Morse, ``Coordination of groups of mobile
  autonomous agents using nearest neighbor rules,'' \emph{IEEE Transactions on
  Automatic Control}, 2003.

\bibitem{cortes1995support}
C.~Cortes and V.~Vapnik, ``Support-vector networks,'' \emph{Machine learning},
  1995.

\bibitem{zhang2013information}
Y.~Zhang, J.~Duchi, M.~I. Jordan, and M.~J. Wainwright, ``Information-theoretic
  lower bounds for distributed statistical estimation with communication
  constraints,'' in \emph{Advances in Neural Information Processing Systems},
  2013.

\bibitem{zhang2013divide}
Y.~Zhang, J.~Duchi, and M.~Wainwright, ``Divide and conquer kernel ridge
  regression,'' in \emph{Conference on Learning Theory}, 2013.

\bibitem{mitliagkas2013memory}
I.~Mitliagkas, C.~Caramanis, and P.~Jain, ``Memory limited, streaming pca,'' in
  \emph{Advances in Neural Information Processing Systems}, 2013.

\bibitem{zhang2012communication}
Y.~Zhang, M.~J. Wainwright, and J.~C. Duchi, ``Communication-efficient
  algorithms for statistical optimization,'' in \emph{Advances in Neural
  Information Processing Systems}, 2012.

\bibitem{zhang2015communication}
Y.~Zhang and L.~Xiao, ``Communication-efficient distributed optimization of
  self-concordant empirical loss,'' \emph{arXiv preprint arXiv:1501.00263},
  2015.

\bibitem{hellman1970learning}
M.~E. Hellman and T.~M. Cover, ``Learning with finite memory,'' \emph{The
  Annals of Mathematical Statistics}, 1970.

\bibitem{forero2010consensus}
P.~A. Forero, A.~Cano, and G.~B. Giannakis, ``Consensus-based distributed
  support vector machines,'' \emph{The Journal of Machine Learning Research},
  2010.

\bibitem{li2014communication}
M.~Li, D.~G. Andersen, A.~J. Smola, and K.~Yu, ``Communication efficient
  distributed machine learning with the parameter server,'' in \emph{Advances
  in Neural Information Processing Systems}, 2014.

\bibitem{wang2013large}
H.~Wang, A.~Banerjee, C.-J. Hsieh, P.~K. Ravikumar, and I.~S. Dhillon, ``Large
  scale distributed sparse precision estimation,'' in \emph{Advances in Neural
  Information Processing Systems}, 2013.

\bibitem{aslan2013convex}
{\"O}.~Aslan, H.~Cheng, X.~Zhang, and D.~Schuurmans, ``Convex two-layer
  modeling,'' in \emph{Advances in Neural Information Processing Systems},
  2013.

\bibitem{romera2013new}
B.~Romera-Paredes and M.~Pontil, ``A new convex relaxation for tensor
  completion,'' in \emph{Advances in Neural Information Processing Systems},
  2013.

\bibitem{tsitsiklis1984problems}
J.~N. Tsitsiklis, ``Problems in decentralized decision making and
  computation.'' DTIC Document, Tech. Rep., 1984.

\bibitem{tsitsiklis1986distributed}
J.~N. Tsitsiklis, D.~P. Bertsekas, and M.~Athans, ``Distributed asynchronous
  deterministic and stochastic gradient optimization algorithms,'' \emph{IEEE
  Transactions on Automatic Control, 1986}.

\bibitem{nedic2010constrained}
A.~Nedic, A.~Ozdaglar, and P.~A. Parrilo, ``Constrained consensus and
  optimization in multi-agent networks,'' \emph{IEEE Transactions on Automatic
  Control}, 2010.

\bibitem{ram2010distributed}
S.~S. Ram, A.~Nedi{\'c}, and V.~V. Veeravalli, ``Distributed stochastic
  subgradient projection algorithms for convex optimization,'' \emph{Journal of
  optimization theory and applications, 2010}.

\bibitem{chen2012diffusion}
J.~Chen and A.~H. Sayed, ``Diffusion adaptation strategies for distributed
  optimization and learning over networks,'' \emph{IEEE Transactions on Signal
  Processing, 2012}.

\bibitem{jakovetic2014fast}
D.~Jakovetic, J.~Xavier, and J.~M. Moura, ``Fast distributed gradient
  methods,'' \emph{IEEE Transactions on Automatic Control, 2014}.

\bibitem{shi2014extra}
W.~Shi, Q.~Ling, G.~Wu, and W.~Yin, ``Extra: An exact first-order algorithm for
  decentralized consensus optimization,'' \emph{arXiv preprint
  arXiv:1404.6264}, 2014.

\bibitem{johansson2009randomized}
B.~Johansson, M.~Rabi, and M.~Johansson, ``A randomized incremental subgradient
  method for distributed optimization in networked systems,'' \emph{SIAM
  Journal on Optimization}, 2009.

\bibitem{duchi2012dual}
J.~C. Duchi, A.~Agarwal, and M.~J. Wainwright, ``Dual averaging for distributed
  optimization: convergence analysis and network scaling,'' \emph{IEEE
  Transactions on Automatic Control}, 2012.

\bibitem{boyd2011distributed}
S.~Boyd, N.~Parikh, E.~Chu, B.~Peleato, and J.~Eckstein, ``Distributed
  optimization and statistical learning via the alternating direction method of
  multipliers,'' \emph{Foundations and Trends{\textregistered} in Machine
  Learning}, vol.~3, no.~1, pp. 1--122, 2011.

\bibitem{eckstein2012augmented}
J.~Eckstein, ``Augmented lagrangian and alternating direction methods for
  convex optimization: A tutorial and some illustrative computational
  results,'' \emph{RUTCOR Research Reports}, vol.~32, 2012.

\bibitem{glowinski1975approximation}
R.~Glowinski and A.~Marroco, ``Sur l'approximation, par e?le?ments finis
  d'ordre un, et la re?solution, par pe?nalisation-dualite? d'une classe de
  proble?mes de dirichlet non line?aires,'' \emph{ESAIM: Mathematical Modelling
  and Numerical Analysis-Mod{\'e}lisation Math{\'e}matique et Analyse
  Num{\'e}rique}, 1975.

\bibitem{gabay1976dual}
D.~Gabay and B.~Mercier, ``A dual algorithm for the solution of nonlinear
  variational problems via finite element approximation,'' \emph{Computers \&
  Mathematics with Applications}, 1976.

\bibitem{bertsekas1988dual}
D.~P. Bertsekas and J.~Eckstein, ``Dual coordinate step methods for linear
  network flow problems,'' \emph{Mathematical Programming}, 1988.

\bibitem{wahlberg2012admm}
B.~Wahlberg, S.~Boyd, M.~Annergren, and Y.~Wang, ``An admm algorithm for a
  class of total variation regularized estimation problems,'' in
  \emph{Preprints of the 16th IFAC Symposium on System Identification}, 2012.

\bibitem{sedghi2014multi}
H.~Sedghi, A.~Anandkumar, and E.~Jonckheere, ``Multi-step stochastic admm in
  high dimensions: Applications to sparse optimization and matrix
  decomposition,'' in \emph{Advances in Neural Information Processing Systems},
  2014.

\bibitem{wang2012online}
H.~Wang and A.~Banerjee, ``Online alternating direction method,'' in
  \emph{Proceedings of the 29th International Conference on Machine Learning
  (ICML-12)}, 2012.

\bibitem{zhang2012efficient}
C.~Zhang, H.~Lee, and K.~G. Shin, ``Efficient distributed linear classification
  algorithms via the alternating direction method of multipliers,'' in
  \emph{International Conference on Artificial Intelligence and Statistics},
  2012.

\bibitem{zhang2014asynchronous}
R.~Zhang and J.~Kwok, ``Asynchronous distributed admm for consensus
  optimization,'' in \emph{Proceedings of the 31st International Conference on
  Machine Learning (ICML-14)}, 2014.

\bibitem{mota2011basis}
J.~F. Mota, J.~Xavier, P.~M. Aguiar, and M.~Puschel, ``Basis pursuit in sensor
  networks,'' in \emph{Acoustics, Speech and Signal Processing (ICASSP), 2011
  IEEE International Conference on}.\hskip 1em plus 0.5em minus 0.4em\relax
  IEEE, 2011, pp. 2916--2919.

\bibitem{schizas2008consensus}
I.~D. Schizas, A.~Ribeiro, and G.~B. Giannakis, ``Consensus in ad hoc wsns with
  noisy linksÑpart i: Distributed estimation of deterministic signals,''
  \emph{Signal Processing, IEEE Transactions on}, vol.~56, no.~1, pp. 350--364,
  2008.

\bibitem{aybat2014alternating}
N.~S. Aybat and G.~Iyengar, ``An alternating direction method with increasing
  penalty for stable principal component pursuit,'' \emph{Computational
  Optimization and Applications}, pp. 1--34, 2014.

\bibitem{aybat2014admm}
N.~S. Aybat, S.~Zarmehri, and S.~Kumara, ``An admm algorithm for clustering
  partially observed networks,'' \emph{arXiv preprint arXiv:1410.3898}, 2014.

\bibitem{ShiYin2014LinADMM}
W.~Shi, Q.~Ling, K.~Yuan, G.~Wu, and W.~Yin, ``On the linear convergence of the
  admm in decentralized consensus optimization,'' \emph{IEEE Transactions on
  Signal Processing}, 2014.

\bibitem{ermin}
S.~Shtern, E.~Wei, and A.~Ozdaglar, ``Distributed alternating direction method
  of multipliers (admm): Performance and network effects,'' in \emph{Working
  paper}.

\bibitem{deng2012global}
W.~Deng and W.~Yin, ``On the global and linear convergence of the generalized
  alternating direction method of multipliers,'' DTIC Document, Tech. Rep.,
  2012.

\bibitem{he20121}
B.~He and X.~Yuan, ``On the o(1/n) convergence rate of the douglas-rachford
  alternating direction method,'' \emph{SIAM Journal on Numerical Analysis},
  2012.

\bibitem{lions1979splitting}
P.-L. Lions and B.~Mercier, ``Splitting algorithms for the sum of two nonlinear
  operators,'' \emph{SIAM Journal on Numerical Analysis}, 1979.

\bibitem{eckstein1998operator}
J.~Eckstein and M.~C. Ferris, ``Operator-splitting methods for monotone affine
  variational inequalities, with a parallel application to optimal control,''
  \emph{INFORMS Journal on Computing}, 1998.

\bibitem{patrinos2014douglas}
P.~Patrinos, L.~Stella, and A.~Bemporad, ``Douglas-rachford splitting:
  complexity estimates and accelerated variants,'' \emph{arXiv preprint
  arXiv:1407.6723}, 2014.

\bibitem{patrinos2014forward}
------, ``Forward-backward truncated newton methods for large-scale convex
  composite optimization,'' \emph{arXiv preprint arXiv:1402.6655}, 2014.

\bibitem{davis2014convergence}
D.~Davis and W.~Yin, ``Convergence rate analysis of several splitting
  schemes,'' \emph{arXiv preprint arXiv:1406.4834}, 2014.

\bibitem{davis2014convergence2}
------, ``Convergence rates of relaxed peaceman-rachford and admm under
  regularity assumptions,'' \emph{arXiv preprint arXiv:1407.5210}, 2014.

\bibitem{giselsson2014diagonal}
P.~Giselsson and S.~Boyd, ``Diagonal scaling in douglas-rachford splitting and
  admm,'' in \emph{IEEE Conference on Decision and Control}, 2014.

\bibitem{nishihara2015general}
R.~Nishihara, L.~Lessard, B.~Recht, A.~Packard, and M.~I. Jordan, ``A general
  analysis of the convergence of admm,'' \emph{arXiv preprint
  arXiv:1502.02009}, 2015.

\bibitem{lessard2014analysis}
L.~Lessard, B.~Recht, and A.~Packard, ``Analysis and design of optimization
  algorithms via integral quadratic constraints,'' \emph{arXiv preprint
  arXiv:1408.3595}, 2014.

\bibitem{bertsekas1989parallel}
D.~P. Bertsekas and J.~N. Tsitsiklis, \emph{Parallel and distributed
  computation: numerical methods}.\hskip 1em plus 0.5em minus 0.4em\relax
  Prentice-Hall, Inc., 1989.

\bibitem{bertsekas2003convex}
D.~P. Bertsekas, A.~Nedi{\'c}, and A.~E. Ozdaglar, \emph{Convex analysis and
  optimization}.\hskip 1em plus 0.5em minus 0.4em\relax Athena Scientific
  Belmont, 2003.

\bibitem{chung1997spectral}
F.~R. Chung, \emph{Spectral graph theory}.\hskip 1em plus 0.5em minus
  0.4em\relax American Mathematical Soc., 1997.

\bibitem{fiedler1973algebraic}
M.~Fiedler, ``Algebraic connectivity of graphs,'' \emph{Czechoslovak
  Mathematical Journal}, 1973.

\bibitem{de2007old}
N.~M.~M. de~Abreu, ``Old and new results on algebraic connectivity of graphs,''
  \emph{Linear algebra and its applications}, 2007.

\end{thebibliography}
\end{document}